\documentclass[11pt]{article}     
\pdfoutput=1 
\usepackage{graphicx}
\usepackage{url}
\usepackage{color}

\usepackage[mathscr]{euscript}		

\usepackage{dsfont}

\usepackage{psfrag}			
\usepackage{hyperref}

\usepackage{amsmath}
\usepackage{amsthm}
\usepackage{mathtools}
\usepackage{amssymb}
\usepackage{bbm} 

\usepackage{caption} 

\usepackage{anysize}

\usepackage{enumerate}
\usepackage{enumitem}

\usepackage{ulem}	
\usepackage{sidecap}    

\usepackage[top=.95in, bottom = 1.02 in, left=1in, right = 1in]{geometry}

\newcommand{\ind}{\mathds}
\newcommand{\floor}[1]{{\lfloor #1 \rfloor}}
\newcommand{\ceil}[1]{{\lceil #1 \rceil}}
\newcommand{\Z}{\ensuremath{\mathbb{Z}}}
\newcommand{\N}{\ensuremath{\mathbb{N}}}
\newcommand{\R}{\ensuremath{\mathbb{R}}}

\newcommand{\E}{\ensuremath{\mathbb{E}}}
\renewcommand{\P}{\ensuremath{\mathbb{P}}}

\newtheorem{theorem}{Theorem}[section]
\newtheorem{lemma}[theorem]{Lemma}

\newtheorem{proposition}[theorem]{Proposition}

\theoremstyle{definition}
\newtheorem{definition}[theorem]{Definition}
\newtheorem{remark}[theorem]{Remark}

\numberwithin{equation}{section}

\definecolor{Red}{rgb}{1,0,0}
\definecolor{Blue}{rgb}{0,0,1}
\definecolor{Olive}{rgb}{0.41,0.55,0.13}
\definecolor{Yarok}{rgb}{0,0.5,0}
\definecolor{Green}{rgb}{0,1,0}
\definecolor{MGreen}{rgb}{0,0.8,0}
\definecolor{DGreen}{rgb}{0,0.55,0}
\definecolor{Yellow}{rgb}{1,1,0}
\definecolor{Cyan}{rgb}{0,1,1}
\definecolor{Magenta}{rgb}{1,0,1}
\definecolor{Orange}{rgb}{1,.5,0}
\definecolor{Violet}{rgb}{.5,0,.5}
\definecolor{Purple}{rgb}{.75,0,.25}
\definecolor{Brown}{rgb}{.75,.5,.25}
\definecolor{Grey}{rgb}{.7,.7,.7}
\definecolor{Black}{rgb}{0,0,0}

\newcommand{\ignore}[1]{{}}

\renewcommand{\P}{\ensuremath{\mathbb{P}}}

\title{Asymptotics for Lipschitz percolation above tilted planes}

\author{
Alexander Drewitz $^{1}$
\and
Michael Scheutzow $^{2}$
\and Maite Wilke-Berenguer $^{2}$
}

\date{\today}

\begin{document}
\normalem

\maketitle
{\footnotesize

\thanks{$^1$ Department of Mathematics, Columbia University, 2990 Broadway, New York City, NY 10027, USA;\\
{\it e-mail:} \href{mailto:drewitz@math.columbia.edu}{drewitz@math.columbia.edu}}

\thanks{$^2$ Institut f\"ur Mathematik, Technische Universit\"at Berlin, MA 7-5, Stra\ss e des 17. Juni 136, 10623 Berlin, Germany;\\
{\it e-mail:} \href{mailto:ms@math.tu-berlin.de}{ms@math.tu-berlin.de}, \href{mailto:wilkeber@math.tu-berlin.de}{wilkeber@math.tu-berlin.de}
}}

\maketitle

\begin{abstract}
We consider  Lipschitz percolation in $d+1$ dimensions above planes tilted by an angle $\gamma$ along one or several coordinate axes.
 In particular, we are interested in the asymptotics of the critical probability 
as $d \to \infty$ as well as $\gamma \uparrow \pi/4.$ Our principal results show that the convergence of the critical probability to $1$
is polynomial as $d\to \infty$ and $\gamma \uparrow \pi/4.$ 
In addition, we identify the correct order of this polynomial convergence and in $d=1$ we also obtain the correct prefactor.
\end{abstract}

\noindent \textbf{Mathematics Subject Classification (2010):} 60K35, 82B20, 82B41, 82B43

\noindent \textbf{Keywords:} Lipschitz percolation, $\rho$-percolation, random surface


\section{Introduction and main results}

The model of Lipschitz percolation was introduced in \cite{DiDoGrHoSc-10}. 
Since its introduction it has been the subject of numerous articles 
and has shown various connections and applications to other topics such as lattice embeddings, plaquette, entanglement and comb percolation or the pinning of interfaces in random media (see e.g.
\cite{GrHo-10}, \cite{DiDoSc-11}, \cite{GrHo-12}, \cite{GrHo-12b}, \cite{HoMa-14}).
In the present article we investigate the critical probability for the existence of a Lipschitz surface of open sites
that lies above a hyperplane which is tilted (along one or several
coordinate axes) by an angle $\gamma$. We are particularly interested in the asymptotics of this critical probability as
$d\to \infty$ and $\gamma \uparrow \pi/4$. An immediate consequence of our results is the existence of non-negative stationary supersolutions to the problem
$$
u_t(x,t) = \Delta u(x,t) + f(x, \bar a\cdot x + u(x,t),\omega) + F
$$
for $\bar a\in(-\alpha,\alpha)^d$ and $F>0$ independent of $\bar a$ for some $\alpha>0$ in the sense of \cite{DiDoSc-11}, i.e., where $f$ describes randomly placed local obstacles. This setting is related to the study of singular homogenization problems, since -- as a cell problem -- it determines the effective velocity $H(\bar a)$ of an interface with slope $\bar a$.

Our context is that of site percolation in $\Z^{d+1}$ with parameter $p \in [0,1]$. That is, $\Omega:= \{0,1\}^{\mathbb Z^{d+1}}$ is the set of configurations and the corresponding probability distribution $\P_p$ is the product measure of Bernoulli distributions with parameter $p$. A site $x \in \mathbb Z^{d+1}$ is called \emph{open} (with respect to $\omega$) if $\omega(x)=1$, and \emph{closed} 
if $\omega(x)=0$.
 
Our main object of study are Lipschitz functions and surfaces defined as follows. A function $F: \Z^d\rightarrow \Z$ is called \emph{Lipschitz} if for any $\bar x,\bar y \in \Z^d$ the implication 
\begin{align*}
 \Vert \bar x-\bar y \Vert_1 =1 \Rightarrow \vert F(\bar x) - F(\bar y)\vert \leq 1
\end{align*}
holds true. We use the term \emph{Lipschitz surface} to refer to a subset of $\Z^{d+1}$ that is the graph of a Lipschitz function. Furthermore, given a realization $\omega \in \Omega$, we
 call the Lipschitz surface \emph{open} if all sites 
 in the Lipschitz surface are open in the sense of site percolation, i.e., if $\omega(\bar x, F(\bar x))= 1$ for all ${\bar x \in \Z^d}$.

It was proven in \cite{DiDoGrHoSc-10} that the event of existence of an open Lipschitz surface completely contained in the upper half-plane $\Z^d\times\N$ undergoes a phase transition. That is, for any dimension $d \ge 1$ there exists a critical probability $p_L(d) \in (0,1)$ such that the following holds: For $p < p_L(d)$ one has that $\P_p$-a.s. there exists no open Lipschitz
surface in $\Z^d\times\N$, whereas for $p > p_L(d)$ one has that $\P_p$-a.s. there exists an open Lipschitz surface in $\Z^d\times\N$.
  Furthermore, an upper bound for $p_L(d)$ and tail estimates for the height of the minimal surface  were established
  for $p$ sufficiently large. These results
  were improved in \cite{GrHo-12}, where in particular exponential tails for the height of the minimal Lipschitz surface 
  have been established for all $p>p_L(d)$. The results were complemented with an asymptotic lower bound yielding $1/d$ as the correct order of magnitude for $1-p_L(d)$. 
  Applications and related results can be found in \cite{DiDoSc-11}, \cite{
GrHo-12b}, \cite{GrHo-10}.

While the investigation of Lipschitz percolation up to now has been focused on  Lipschitz surfaces that stay above the plane
$L:=\Z^d\times\{0\}$, we are interested in the effect of \lq tilting\rq\, this plane. To make this more precise let us define  
 for any $d \in \N, \, \alpha \in [0,1)$ and $\eta \in \{-1,0,+1\}^d$ the \emph{tilted planes}
 \begin{align*}
  L^{\alpha,d}_{\eta} := \Big\{ (x_1, \ldots, x_{d+1}) \in \Z^{d+1} \mid x_{d+1} = 
  \Big \lfloor \alpha \sum_{i=1}^d \eta_ix_i \Big \rfloor\Big\}.
 \end{align*}
For computational convenience we introduce the parameter $\alpha$ as in the above definition, instead of directly working with the angle $\gamma $ by which a plane is tilted along all the coordinate axes in direction $e_i$ for which $\eta_i = 1$, $1 \le i \le d,$
in the above choice of $\eta$ (and $-\gamma$ in the case that $\eta = -1$). However, given $\eta,$ there is a natural
 one-to-one correspondence between $\alpha$ and 
the angle $\gamma$. Also, note that the case of $\alpha =0$ as well as the case $\eta = 0$ correspond to $\gamma = 0$ and thus to standard Lipschitz percolation.
The restriction to $\alpha \in [0,1)$, resp. $\gamma < \pi/4$, is natural, once one realizes that for 
$\eta \ne 0$, $\alpha \geq 1$ (resp. $\gamma \geq \pi/4$), and any $p<1$, $\P_p$-a.s. there exists no open Lipschitz surface above the plane $L^{\alpha,d}_{\eta}$. 

In the study of Lipschitz percolation above tilted planes, the related concept of Lipschitz percolation above \lq inverted pyramids\rq\ turns out to be helpful. Thus, we introduce for any $d \in \N, \, \alpha \in [0,1)$ and $\eta \in \{-1,0,+1\}^d$ the 
\emph{inverted pyramid} $\nabla_\eta^{\alpha, d}$ as
\begin{align*}
& \nabla^{\alpha, d}_\eta 
:=
	\Big \{(x_1, \dots, x_{d+1}) \in \mathbb Z^{d+1} \mid x_{d+1} = 
	\max_{\substack{\eta' \in \{-1,0,+1\}^d\\ \Vert \eta' \Vert_1 = \Vert \eta \Vert_1}} 
	\Big \{ \Big \lfloor \alpha \sum_{i=1}^d \eta_i'  x_i \Big \rfloor \Big \} \Big \}.
\end{align*}
We can now formulate our main result:

\begin{theorem}
 \label{thm:mainthm_intro}
 There exists a phase transition for both Lipschitz percolation above planes and
 Lipschitz percolation above inverted pyramids, and their critical probabilities coincide. This critical probability $p_L(\alpha,d,\eta)$ is nontrivial and depends on $\eta$ only via
 $\Vert \eta \Vert_1$. Furthermore,
 \begin{equation}\label{eq:mr_d}
 1-p_L(\alpha, d, \eta) \asymp d^{-\frac{1}{1-\alpha}}, \qquad \text{ as } d \rightarrow \infty, 
 \end{equation}
 and
\begin{equation}\label{mr_alpha}
 1-p_L(\alpha, d, \eta) \asymp (1-\alpha)^d, \qquad \text{ as }\alpha \rightarrow 1. 
 \end{equation}
\end{theorem}
Here we write $f(s) \asymp g(s)$ as $ s\rightarrow \bar s$ for two functions $f$ and $g$
 if there exist positive and finite constants $c,C$ such that $\liminf_{s \rightarrow \bar s} f(s)/g(s) \geq c$ and $\limsup_{s \rightarrow \bar s} f(s)/g(s) \leq C$.

For the reader's convenience, Theorem \ref{thm:mainthm_intro} is a concise
summary of the principal asymptotics for $p_L(\alpha, d, \eta)$ obtained in this article. The actual asymptotics
 we obtain are more precise and will be given as individual results below.

The article is structured as follows.
Section \ref{sec:Lipschitz} is concerned with general results on Lipschitz percolation in the set-up of tilted planes.
 Proposition \ref{prop:equiv} establishes the non-trivial phase transition for $p_L(\alpha, d, \eta)$,
  whereas Lemma \ref{lem:mon} exposes the monotonicity relations for the individual parameters.

Section \ref{sec:bounds} outlines all bounds on the critical probabilities separated into two subsections, one for lower and one for upper bounds. Using the notation of \eqref{eq:qDef}, the asymptotics
\eqref{eq:mr_d} and \eqref{mr_alpha} follow by combining Propositions \ref{prop:allAlphaUBp} and \ref{prop:UBqAsympD}, 
as well as Propositions \ref{prop:lbq_alpha} and \ref{prop:UBqAsympAlpha}, respectively. As explained in Proposition \ref{prop:D1}, for $d=1$ we obtain the exact asymptic behavior for $\alpha \rightarrow 1$. In addition, Proposition \ref{prop:qLBkDepOnd} provides lower bounds for the critical probabilities, depending on how the number of tilted axes behaves asymptotically with the dimension. 

The corresponding proofs and further auxiliary results are contained in Section \ref{sec:proofs}.


\section{Further notation and auxiliary results} \label{sec:Lipschitz}

We begin by defining the events to be considered and to this end denote by $L^{\alpha, d}_{\eta, \geq}$ the upper half space above $L^{\alpha, d}_{\eta}$.
For this purpose, denote the set of all Lipschitz functions by $\Lambda$.

\begin{definition}
 Let $\mathsf{LIP}^{\alpha, d}_{\eta}$ denote the event that there exists an open Lipschitz surface 
contained in  $L^{\alpha, d}_{\eta,\ge}$, i.e.,
	\[\mathsf{LIP}^{\alpha, d}_{\eta} := 
	\Big \{ \omega \in \Omega \mid \exists F \in \Lambda : \forall \bar{x} \in \mathbb Z^{d}: \omega((\bar{x}, F(\bar{x})))=1 \text{ and } F(\bar{x}) > \Big \lfloor \alpha \sum_{i=1}^d \eta_i \bar{x}_i \Big \rfloor \Big \}.\]
Similarly to the case of planes we use $\mathsf{LIP}(\nabla^{\alpha, d}_{\eta})$ to denote the event of existence of a Lipschitz surface above the inverted pyramid $\nabla_\eta^{\alpha, d}$, i.e.,
\begin{align*}
&\mathsf{LIP}(\nabla^{\alpha, d}_{\eta}):= \\
& \Big \{\omega \in \Omega \mid \exists F \in \Lambda  :  \forall \bar{x} \in \mathbb Z^{d} :  \omega((\bar{x}, F(\bar{x})))=1 
	\text{ and } F(\bar{x}) >  
	\max_{\substack{\eta' \in \{-1,0,+1\}^d\\ \Vert \eta' \Vert_1 = \Vert \eta \Vert_1}} 
	 \Big \{ \Big\lfloor \alpha \sum_{i=1}^d \eta_i'  \bar x_i \Big \rfloor \Big \}   \Big\}.
	\end{align*}
\end{definition}

\begin{proposition}
 \label{prop:equiv}
For any $d\ge 1,$   $\alpha \in [0,1)$ and $\eta \in \{-1,0,+1\}^d$, there exists a critical probability 
 $p_L(\alpha,d, \eta) \in (0,1)$ 
such that
\begin{align}\label{eq:planePhaseTrans}
&  \P_p(\mathsf{LIP}(\nabla^{\alpha, d}_{\eta})) = \mathbb P_p(\mathsf{LIP}^{\alpha, d}_{\eta}) = \begin{cases}
	0, \;   p\in [0,p_L(\alpha,d, \eta)),\\
        1, \;  p \in (p_L(\alpha,d,\eta ),1].
        \end{cases}
\end{align}
In fact, for any $\eta' \in \{-1,0,1\}$ with $\Vert \eta \Vert_1 = \Vert \eta'\Vert_1$,
\begin{align} \label{eq:pNormDep}
p_L(\alpha,d, \eta) = p_L(\alpha,d, \eta' ).
\end{align}
Therefore, $p_L(\alpha,d, \eta)$ depends on $\eta$ only through the number of nonzero entries.
\end{proposition}

This means that there exists a phase transition for both Lipschitz percolation above tilted planes and above inverted pyramids, and their critical probabilities coincide.
Due to \eqref{eq:pNormDep} it is convenient to define $p_L(\alpha, d,k) := p_L(\alpha,d,\eta)$ for any $\eta \in \{-1,0,+1\}$ such that $\Vert \eta \Vert_1 = k\in\{0,\ldots,d\}$. Furthermore, we set
\begin{equation} \label{eq:qDef}
q_L(\alpha, d,k) := 1-p_L(\alpha, d,k).
\end{equation}
For notational convenience we will formulate most of our results for $q_L$ instead of $p_L$ since the latter usually tends to $1$ and hence the former to $0$.

\begin{proof}[Proof of Proposition \ref{prop:equiv}]
First observe that due to the symmetries of $\Z^d$ and the i.i.d.-product structure of $\P_p$,  the quantity $\P_p(\mathsf{LIP}^{\alpha, d}_{\eta})$ depends on $\eta$ only through  $\Vert \eta \Vert_1$. Thus, if the postulated critical probabilities exist,
then they must fulfill \eqref{eq:pNormDep}.

We now start with showing the second equality in  \eqref{eq:planePhaseTrans} for some $p_L(\alpha,d, \eta) \in [0,1]$.
Since $\mathsf{LIP}^{\alpha, d}_{\eta}$  is an 
increasing event,  it is immediate that $\mathbb P_p(\mathsf{LIP}^{\alpha, d}_{\eta})$ is nondecreasing in $p.$
Therefore, it is sufficient to show 
 that it takes values in $\{0,1\}$ only.

Define the shift  $\theta: \omega \mapsto \omega(\cdot, \ldots, \cdot +1)$ 
in the $(d+1)$-st coordinate. Then $\theta$ is measure preserving for $\P_p$ and ergodic with respect to $\P_p.$ As a consequence, since $\theta^{-1}(\mathsf{LIP}^{\alpha, d}_{\eta}) \subset \mathsf{LIP}^{\alpha, d}_{\eta}$ and $\P_p(\theta^{-1}(\mathsf{LIP}^{\alpha, d}_{\eta}))=\P_p(\mathsf{LIP}^{\alpha, d}_{\eta})$, the event $\mathsf{LIP}^{\alpha, d}_{\eta}$ is $\P_p$-a.s. invariant with respect to $\theta$, i.e. $\P_p(\mathsf{LIP}^{\alpha, d}_{\eta} \triangle \theta^{-1}(\mathsf{LIP}^{\alpha, d}_{\eta}))=0,$ and by Proposition 6.15 in \cite{Breiman} this already implies 
\begin{align*}
 \P_p (  \mathsf{LIP}^{\alpha, d}_{\eta} )\in \{0,1\}.
\end{align*}
This establishes the second equality in  \eqref{eq:planePhaseTrans} for some $p_L(\alpha,d, \eta) \in [0,1]$.

In order to obtain the first equality of \eqref{eq:planePhaseTrans}, due to the 
second
equality in \eqref{eq:planePhaseTrans} and
$\mathsf{LIP}(\nabla^{\alpha, d}_{\eta}) \subseteq  \mathsf{LIP}^{\alpha, d}_{\eta},
$
it remains to show
that
$
 \mathbb P_p(\mathsf{LIP}^{\alpha, d}_{\eta}) = 1$ implies
$\P_p(\mathsf{LIP}(\nabla^{\alpha, d}_{\eta})) = 1.$
By symmetries, $ \mathbb P_p(\mathsf{LIP}^{\alpha, d}_{\eta}) = 1$
 already yields
\begin{equation*} 
\P_p \Big( \bigcap_{\substack{\eta' \in \{-1,0,+1\}^d\\ \Vert \eta' \Vert_1 = \Vert \eta \Vert_1}} \mathsf{LIP}^{\alpha, d}_{\eta'} \Big) = 1.
\end{equation*}
Note that the pointwise maximum of Lipschitz functions is a Lipschitz function again and thus 
 \begin{align*}
  \bigcap_{\substack{\eta' \in \{-1,0,+1\}^d\\ \Vert \eta' \Vert_1 = \Vert \eta \Vert_1}} \mathsf{LIP}^{\alpha, d}_{\eta'}  \subseteq \mathsf{LIP}(\nabla^{\alpha, d}_{\eta}).
 \end{align*}
 Thus \eqref{eq:planePhaseTrans} holds true.
 
 It remains to show the nontriviality of the phase transition, i.e., that $p_L(\alpha,d, \eta) \in (0,1).$
Proposition \ref{prop:allAlphaUBp} below in particular shows that $p_L(\alpha, d, d) < 1$ for all $\alpha \in [0,1)$ and $d \ge 1$;  hence, using \eqref{eq:etaMon} below, we deduce $p_L(\alpha, d, k) < 1$ for all $0 \le k \le d.$ On the other hand, $p_L(\alpha, d, k) > 0$ for all $0 \le k \le d$ follows from the fact that the critical probability for the existence of an infinite connected component in the $1$-norm
in $(d+1)$-dimensional Bernoulli site-percolation (which is a lower bound for $p_L(\alpha, d, k)$) is strictly positive.
\end{proof}

Using the above result one can obtain  some simple but helpful monotonicity results for the critical probabilities. 

\begin{lemma} \label{lem:mon}
 For all $d \in \mathbb N,$ and $\alpha, \alpha' \in [0,1)$ such that $\alpha \leq \alpha'$, we have 
\begin{equation} 
\forall\, k = 0, \ldots, d: \quad p_L({\alpha},d, k) \leq p_L({ \alpha'}, d, k), \label{eq:alphaMon}
\end{equation}
\begin{equation}
 \forall\, k = 0, \ldots, d:\quad p_L({\alpha}, d, k)
 \leq p_L({\alpha}, d+1, k), \label{eq:dMon}
 \end{equation}
 and
 \begin{equation}
\forall\, k = 0, \ldots, d-1:\quad p_L({\alpha},d, k) \leq p_L({\alpha}, d, k+1).  \label{eq:etaMon} 
\end{equation}
\end{lemma}

\begin{proof}
We start by proving the monotonicity in $\alpha$, which is best seen considering Lipschitz surfaces above inverted pyramids. Note that for $\alpha' \geq \alpha$, one has $ \nabla^{\alpha',d}_\eta \geq \nabla^{\alpha, d}_\eta,$ in the sense that for any $(\bar y, y_{d+1}^{\alpha'}) \in \nabla^{ \alpha',d}_\eta$ and $(\bar y, y_{d+1}^{\alpha}) \in \nabla^{\alpha,d}_\eta$ we have $y^{ \alpha'}_{d+1} \geq y^{\alpha}_{d+1}$. Hence $\mathsf{LIP}(\nabla^{\alpha', d}_{\eta}) \subseteq
\mathsf{LIP}(\nabla^{\alpha, d}_{\eta}),$ which implies \eqref{eq:alphaMon}.

On the other hand, to prove \eqref{eq:dMon} choose $\eta \in \{-1,0,+1\}^{d+1}$ with $\Vert \eta \Vert_1 = k,$ and let $1 \le j \le
d+1$ be such that $\eta_j = 0.$
Then \eqref{eq:dMon}
 follows directly from the fact that the cross section of a Lipschitz surface in $L^{\alpha, d+1}_{\eta, \ge}$ with $\Z^{j-1} \times
\{0\}\times\mathbb Z^{d-j+1}$ mapped to $\Z^d$ by eliminating the $j$-th coordinate
is again a Lipschitz surface contained in  $L^{\alpha, d}_{\eta^{(j)}, \ge}$, for $\eta^{(j)}:=(\eta_1, \ldots, \eta_{j-1},\eta_{j+1}, \ldots, \eta_{d+1})$, combined with the fact that $\Vert \eta^{(j)} \Vert_1 = k$ and \eqref{eq:pNormDep}.

Lastly, \eqref{eq:etaMon} follows from the fact that for any $1\le j \le d,$ $ \nabla^{\alpha, d}_{\eta_{j \to 0}} \geq \nabla^{\alpha, d}_\eta$ in the above sense and thus $\mathsf{LIP}(\nabla^{\alpha, d}_{\eta_{j \to 0}}) \supset \mathsf{LIP}(\nabla^{\alpha, d}_{\eta})$,
where $\eta_{j \to 0}$ is obtained from $\eta$ by replacing the $j$-th coordinate by $0.$
\end{proof}


\section{Bounds on the Critical Probabilities} \label{sec:bounds}

For functions $f,g$ we write $f(s) \lesssim g(s)$ as $s\rightarrow \bar s$, if $\limsup_{s \rightarrow \bar s} f(s)/g(s) \leq 1$, 
we write $f(s) \gtrsim g(s)$ as $s\rightarrow \bar s$, if $\liminf_{s \rightarrow \bar s} f(s)/g(s) \geq 1$, and
asymptotic equivalence is denoted by
 $f(s) \sim g(s), \,s\rightarrow \bar s$ (i.e., if $f(s) \lesssim g(s)$ and $f(s) \gtrsim g(s)$ as $s\rightarrow \bar s$). With this notation we can write the results on the bounds in \cite{GrHo-12} as
\begin{align} \label{eq:GH12}
\begin{split}
 q_L(0,d,0) &\geq (8d)^{-1}, \quad \text{ for all } d \in \N,\\
 q_L(0,d,0) &\lesssim (2d)^{-1}, \quad \text{ as } d\rightarrow \infty.
 \end{split}
\end{align}

\subsection{Lower Bounds for $q_L(\alpha, d, k)$ } 


\begin{proposition}[General bound]\label{prop:allAlphaUBp}
 For any $d \geq 1$ and $\alpha \in [0,1)$ one has
 \begin{align*}
  q_L(\alpha, d, d) \geq \frac{1}{2} (4d)^{-\frac{1}{1-\alpha}}.
 \end{align*}
\end{proposition}


Note that for $\alpha =0$ this is exactly the lower bound of \eqref{eq:GH12}.
In a similar way one can find bounds for the critical probability in the case that the number $k$ of axes along which the plane is tilted depends on the dimension $d$:
\begin{proposition} \label{prop:qLBkDepOnd} 
 Consider a function $\varphi : \N \to \N_0$ with $\varphi(d) \le d$ for all $d \in \N$. 
\begin{enumerate}
 \item If for some $\alpha \in [0,1)$ one has that $\varphi (d) \in o(d^{1-\alpha})$ as $d\rightarrow \infty$, then
\begin{align*}
 q_L(\alpha,d,\varphi(d))  \gtrsim\frac{1}{8} d^{-1}, \quad \text{ as }d\rightarrow \infty.
\end{align*}
 \item If for some
 $\alpha \in [0,1)$ and $c \in [0,1]$ one has $\varphi(d) \sim cd^{1-\alpha}$ as $d\rightarrow \infty$, 
 then there exists a constant $C(c,\alpha)>0$ such that 
\begin{align*}
  q_L(\alpha,d,\varphi(d))  \gtrsim C(c,\alpha) d^{-1}, \quad \text{ as }d\rightarrow \infty.
\end{align*}
 \item If for some
  $c \in (0,1]$ one has
 $\varphi(d) \sim cd$ as $d \rightarrow \infty$, then for $\alpha \in (0,1)$,
\begin{align*}
 q_L(\alpha, d, \varphi(d)) \gtrsim \frac{1}{4} (1-\alpha)(cd)^{-\frac{1}{1-\alpha}}, \quad \text{ as }d\rightarrow \infty.
\end{align*}
\end{enumerate}
\end{proposition}

\begin{remark}
 The constant in Proposition \ref{prop:qLBkDepOnd}, (b), satisfies $C(c,0) = C(0,\alpha)= 1/8$ for any $c \in [0,1]$, $\alpha \in [0,1)$;
 this is what one would hope for, given that these cases correspond to standard Lipschitz percolation.
 
 The bound in Proposition \ref{prop:qLBkDepOnd}, (c), is an improvement compared to Proposition \ref{prop:allAlphaUBp} at the expense of being of asymptotic nature only. 
\end{remark}


\begin{proposition}[]\label{prop:lbq_alpha}
 For each $d \geq 1$ and each $k = 1,\ldots,d$ there exists a constant $C(k,d)>0$ such that for all $\alpha \in [0,1)$ one has
 \begin{align*}
  q_L(\alpha,d,k) \geq C(k,d)(1-\alpha)^k.
 \end{align*}
 \end{proposition}
 
 \begin{proposition}\label{prop:D1}
 For $d=1$ one has $q_L(\alpha,1,1) \gtrsim (1-\alpha)$ as $\alpha \rightarrow 1$, which together with Proposition \ref{prop:UBqAsympAlpha}  below yields
 \begin{align*}
  q_L(\alpha,1,1) \sim (1-\alpha), \qquad \text{ as } \alpha \rightarrow 1.
 \end{align*}
 \end{proposition}
 
\subsection{Upper Bounds for $q_L(\alpha, d, k)$} 
 

 \begin{proposition}[Asymptotic behavior for $d \rightarrow \infty$]\label{prop:UBqAsympD}
  For every $\alpha \in [0,1)$ there exists a constant $C(\alpha)$ such that
  \begin{align*}
   q_L(\alpha,d,d) \lesssim C(\alpha) d^{-\frac{1}{1-\alpha}}, \qquad \text{ as } d\rightarrow \infty.
  \end{align*}
  More precisely, $C(\alpha) = \theta^{\frac{1}{1-\alpha}}/({\rm e}^{\theta} -1)$, where $\theta$ is the unique solution to 
  $\theta {\rm e}^{\theta}/({\rm e}^{\theta} -1) = 1/(1-\alpha)$ and $C(0) = 1$.
 \end{proposition}


 \begin{proposition}[General bound]\label{prop:UBqAsympAlpha}
  For any $\alpha \in [0,1)$ and $d \in \N$ 
  \begin{align*}
  q_L(\alpha,d,d) \leq \frac{d!(1-\alpha)^d}{1+d!(1-\alpha)^d} \leq d!(1-\alpha)^d.
  \end{align*}
 \end{proposition}
 \begin{remark}
Since $q_L(\alpha, d,k) \leq q_L(\alpha, k,k)$ by Lemma \ref{lem:mon}, Proposition \ref{prop:UBqAsympAlpha}
immediately implies upper bounds for $q_L(\alpha, d, k)$ for any $k=1,\ldots,d$ also.
 \end{remark}


\section{Proofs} \label{sec:proofs}

As explained in \cite{DiDoGrHoSc-10} and \cite{GrHo-12} for standard Lipschitz percolation, the lowest open Lipschitz surface (above $L^{\alpha,d}_{\eta}$) may be constructed as a blocking surface to a certain type of paths called (admissible) $\lambda$-paths. This characterization is the core of the proofs in this section.

Denote by $e_1, \ldots, e_{d+1} \in \Z^{d+1}$ the standard basis vectors of $\Z^{d+1}$.

\begin{definition} \label{def:adLamPa}
 For $x,y \in \Z^{d+1}$ a \emph{$\lambda$-path} from $x$ to $y$ is any finite sequence $x=u_0, \ldots, u_n=y$ 
 of distinct sites in $\Z^{d+1}$ such that for all $ i = 1, \ldots,n$
 \begin{align*}
    u_i - u_{i-1} \in \{e_{d+1}\} \cup \{-e_{d+1} \pm e_{j} \mid j = 1, \ldots, d\}.
 \end{align*}
 Such a path will be called \emph{admissible (with respect to $\omega$)}, if for all $i = 1, \ldots, n$ the following implication holds:
 \begin{align*}
  \text{ If } u_i-u_{i-1} = e_{d+1}, \text{ then } u_i \text{ is closed (with respect to $\omega$).}
 \end{align*}
\end{definition}
For any $x,y \in \Z^{d+1}$ denote by $x \rightarrowtail y$ the event that there exists an admissible $\lambda$-path from $x$ to $y$. 
We then define for all $ x \in \Z^d$,  $\alpha \in [0,1)$, $d \in \N$ and $\eta \in \{-1,0,+1\}$  the function
\begin{align} \label{eq:LSfromlambda}
\; F^{\alpha, d}_{\eta}(\bar x):= \sup\{n \in \Z \mid \exists y \in L^{\alpha, d}_{\eta}:\; y \rightarrowtail (\bar x, n) \} + 1.
\end{align}
Note that the graph of $F$ is contained in $L^{\alpha,d}_{\eta}$. As in \cite{DiDoGrHoSc-10} and  \cite{GrHo-12}, it is easy to see that  
the function defined in \eqref{eq:LSfromlambda} 
describes a Lipschitz function whose graph consists of open sites,
if and only if it is finite  for all $\bar x \in \Z^d$. This in turn holds true if and only if it is finite at $\bar x = 0$. Thus, in the analysis of the existence of an open Lipschitz surface we can focus on the behavior of $F^{\alpha, d}_{\eta}(0)$ as defined above.

It will be useful to define $L^{\alpha, d}_{\eta}(h) := L^{\alpha,d}_{\eta} + he_{d+1}$ and denote by $\mathcal L^{\alpha,d}_{\eta}(h)$ the random set of sites in $L^{\alpha, d}_{\eta}(h)$ reachable by an admissible $\lambda$-path started in the origin.
We have taken the practice of marking elements of $\Z^d$ with a bar as in $\bar x \in \Z^d$ in order to distinguish them from  canonical elements $x \in \Z^{d+1}$. In the same vein, for $x = (x_1, \ldots, x_{d+1})\in \Z^{d+1}$, we use $\bar x$ to refer to $(x_1, \ldots, x_d)$ as well as $(\bar x,x_{d+1})$ to denote $x$. In addition, by a slight abuse of notation we use $0$ to denote the origin of $\Z, \Z^d$ and $\Z^{d+1}$. As we have tacitly done above already,
 it will be necessary to distinguish between $\N$ and $\N_0$. For a set $A$ we will use $\vert A \vert$ to denote its cardinality. 

In addition, due to the symmetries of $\Z^d$ and the product structure of $\P_p$, we will w.l.o.g.
 from now on assume that for any $k = 1, \ldots, d$, the vector $\eta$ is of the form
\begin{align*}
 \eta = (\underbrace{1, \ldots,1}_{ k \text{ times}},\underbrace{0,\ldots,0}_{d-k \text{ times}}).
\end{align*}

\subsection{Lower Bounds for $q_L(\alpha, d, k)$} 


We begin with a criterion ensuring the existence of an open Lipschitz surface by providing suitable conditions for the $\P_p$-a.s. finiteness of $F^{\alpha, d}_{\eta}$ as defined in \eqref{eq:LSfromlambda}.
\begin{lemma}[Criterion for existence of an open Lipschitz surface]\label{lem:surfaceExistence}
Let $F^{\alpha, d}_{\eta}$ be defined as in \eqref{eq:LSfromlambda}. Then, for any $\bar x \in \Z^d$ and $h \in \N$,
\begin{equation} \label{eq:tailEstCond}
\P_p \Big ( F^{\alpha,d}_{\eta}(\bar{x}) - \Big \lfloor \alpha \sum_{i=1}^d \eta_i \bar{x}_i \Big \rfloor \ge h \Big) \le  \E_p [\vert \mathcal L_{\eta}^{\alpha, d}(h-2) \vert ].
\end{equation} 
 In particular, if
 \begin{align} \label{eq:summability}
  \lim_{h \rightarrow \infty} \E_p[\vert \mathcal L^{\alpha,d}_{\eta}(h) \vert ] = 0,
 \end{align}
then 
\begin{equation} \label{eq:surfaceEx}
\P_p(\mathsf{LIP}^{\alpha, d}_{\eta})  = 1.
\end{equation}
\end{lemma}


\begin{proof}[Proof of Lemma \ref{lem:surfaceExistence}]
In order to prove \eqref{eq:tailEstCond} we start by observing that 
 for every $\bar{x} \in \mathbb Z^{d},$ 
\begin{align} \label{eq:stochDom}
\text{the random variable }F^{\alpha,d}_{\eta}(0)+1 \text{ stochastically dominates }
F^{\alpha,d}_{\eta}(\bar{x}) - \Big \lfloor \alpha \sum_{i=1}^d \eta_i \bar{x}_i \Big \rfloor,
\end{align}
where the $+1$ stems from lattice effects.
Now we estimate 
\begin{align*}
 \P_p(F^{\alpha,d}_{\eta}(0)  \ge h + 1) & =  \mathbb P_p 	\Big( \exists z \in L^{\alpha, d}_{\eta}\,: \, z \rightarrowtail (0, h) \Big)  
		\leq \sum_{z \in L^{\alpha, d}_{\eta}} \mathbb P_p ( z \rightarrowtail (0,h) )\\
		 & \le	\sum_{z \in L^{\alpha, d}_{\eta}(h)} \mathbb P_p  ( 0 \rightarrowtail z ) = \E_p[\vert \mathcal L_{\eta}^{\alpha, d}(h) \vert ]. 
\end{align*}
In combination with \eqref{eq:stochDom}, this supplies us with 
\eqref{eq:tailEstCond} which finishes the proof. Note that we used the fact that if a site 
 $x = (\bar x, h)$ with $h \ge \Big \lfloor \alpha \sum_{i=1}^d \eta_i \bar{x}_i \Big \rfloor$ is reachable from $L^{\alpha, d}_{\eta}$ by an admissible $\lambda$-path, then so is any site 
$x = (\bar x, i)$ with $\Big \lfloor \alpha \sum_{i=1}^d \eta_i \bar{x}_i \Big \rfloor \le i \le h$. This stems from the observation that if we remove the last step the admissible $\lambda$-path took in the upward direction and then trace it, we obtain again an admissible $\lambda$-path reaching the site right below $x$.

The fact that \eqref{eq:summability} implies \eqref{eq:surfaceEx} follows immediately from \eqref{eq:tailEstCond} in combination with
the
observation below \eqref{eq:LSfromlambda}.
\end{proof}


The common core of the proofs of Propositions \ref{prop:allAlphaUBp} and  \ref{prop:qLBkDepOnd} can be summarized in the following, somewhat technical lemma.

\begin{lemma}[A general lower bound]\label{lem:LBwithPi}
 Let $\alpha \in [0,1)$, $d \in \N$ and $k = 0, \ldots, d$. Then for any choice of
 \begin{align} \label{eq:pchoice}
   p_1, p_2, p_3, p_4 \in (0,1) \text{ such that } \sum_{i=1}^4 p_i = 1
 \end{align}
 we obtain 
\begin{align}\label{eq:LBqWithp_i}
 q_L(\alpha, d,k) \geq \min\left\{\frac{1}{k} p_1 \sqrt{p_2 p_3}\, ,\, p_1 \Big( \frac{p_3}{k}\Big)^{\frac 1{1-\alpha}}\, ,\, \frac{p_1p_4}{2(d-k)}\right\}.
\end{align}
\end{lemma}
 Note that the above holds true for all possible choices of our parameters -- in particular for $ k \in \{0,d\}$ -- if we use the convention of $1/0=\infty$. This somewhat unelegant agreement may be justified in this case as it avoids  the need of repeating analogous 
 computations without the respective terms.

 
 \begin{proof}[Proof of Lemma \ref{lem:LBwithPi}]
  In order to obtain the existence of an open Lipschitz surface and thus the lower bound through Lemma \ref{lem:surfaceExistence}, we will show the following estimate under appropriate  assumptions on $q = 1-p$:
 
For $d \ge 1,$ $\alpha \in [0,1)$, $k = 1, \ldots, d$ and  $q$
smaller than the right-hand side of \eqref{eq:LBqWithp_i},
 there exist constants $\delta \in (0,1)$ and $C>0$ such that for all $h\in \N,$
\begin{equation} \label{eq:tailEst}
   \E_p [\vert \mathcal L_{\eta}^{\alpha, d}(h) \vert ] \leq C\delta^{ h-1}.
 \end{equation}
 
We will say that the $j$-th step of a $\lambda$-path $(u_n)$ is \emph{positive downward}, if $u_j-u_{j-1} \in \{-e_{d+1} + e_l \mid l = 1, \ldots, k\}$ and \emph{negative downward} if $u_j-u_{j-1} \in \{-e_{d+1} - e_l \mid l = 1, \ldots, k\}$ and use $D^+=D^+(u)$, resp $D^-=D^-(u)$ to denote the number of these steps. In analogy, $D=D(u)$ will denote the number of \emph{downward} steps such that $u_j-u_{j+1} \in \{-e_{d+1} \pm e_l \mid l = k+1, \ldots, d\}$ and $U = U(u)$ will be the number of \emph{upward} steps, i.e., those for which $u_j -u_{j-1} = e_{d+1}$.

Now for any natural numbers $U,$ $D^+,$ $D^-$ and $D$, the number of  $\lambda$-paths 
starting in the origin with $U$ upward steps as well as $D^+$ positive, $D^-$ negative and $D$ neutral downward steps, respectively, can be estimated from above by
\[\binom{U+D^++D^-+D}{U,D^+,D^-,D}k^{D^++D^-}(2(d-k))^D.\]
Thus the expected number of such paths which are admissible can  be upper bounded by
\begin{equation} \label{eq:expNrPathUB}
\binom{U+D^++D^-+D}{U,D^+,D^-,D}k^{D^++D^-}(2(d-k))^Dq^U.
\end{equation}
In addition, due to the multinomial theorem, for any $p_1,\ p_2,\ p_3,\ p_4$ chosen as in \eqref{eq:pchoice} we have
\begin{align*}
{\binom{U+D^++D^-+D}{U,D^+,D^-,D} p_1^U p_2^{D^+} p_3^{D^-} p_4^D \le 1,}
\end{align*}
and hence
\begin{align} \label{eq:multinomEst}
\binom{U+D^++D^-+D}{U,D^+,D^-,D} \le \Big(\frac{1}{p_1}\Big)^U \Big(\frac{1}{p_2}\Big)^{D^+} \Big(\frac{1}{p_3}\Big)^{D^-}\Big(\frac{1}{p_4}\Big)^D.
\end{align}
In order to simplify notation, note that the \lq best strategy\rq \ for admissible $\lambda$-paths is to go for the  negative orthant in the first $d$ coordinate axes, in the sense that
\begin{equation*} 
\sum_{y \in L^{\alpha, d}_{\eta}(h)}	\mathbb P_p(0 \rightarrowtail y) 	
 \le 2^d \sum_{y \in L^{\alpha, d}_{\eta}(h) \cap ((-\N_0)^d \times \Z)}	\mathbb P_p(0 \rightarrowtail y) 	.
\end{equation*}
 Since at each downward step of a $\lambda$-path the $(d+1)$-st coordinate of the path is decreased by one, the total number $U(u)$ 
 of upward steps of a $\lambda$-path $(u_n)$  starting in 0  and ending in $L^{\alpha, d}_{\eta}(h) \cap ((-\N_0)^d \times \Z)$ fulfills
\begin{align*}  
U(u) & = D^+(u) + D^-(u) + D(u) +  \floor{\alpha (D^+(u)-D^-(u))} + h
\end{align*}
and
$$
D^+(u)-D^-(u)  \leq 0.
$$
Using \eqref{eq:expNrPathUB} and \eqref{eq:multinomEst} and choosing $q <p_1$ 
we can thus estimate
\begin{align}
 \sum_{y \in L^{\alpha, d}_{\eta}(h) \cap ((-\N_0)^d \times \Z)}  \mathbb P_p(0 \rightarrowtail y)\notag\\
 &	\hspace{-65pt} \leq \sum_{\substack{D^+,D^-,D\geq 0\,:\,D^+-D^-\leq 0\\U =  h+D^++D^-+D +\floor{\alpha(D^+-D^-)} }}\left(\frac{q}{p_1}\right)^U\left(\frac{k}{p_2}\right)^{D^+}\left(\frac{k}{p_3}\right)^{D^-}\left(\frac{2(d-k)}{p_4}\right)^{D}\notag\\
													& \hspace{-65pt} \leq \sum_{\substack{D^+, D^-, D \geq 0,\\ D^- - D^+ \geq 0}} \left(\frac{q}{p_1}\right)^{ D^++D^-+D+\floor{\alpha(D^+-D^-)}+h}\left(\frac{k}{p_2}\right)^{D^+}\left(\frac{k}{p_3}\right)^{D^-}\left(\frac{2(d-k)}{p_4}\right)^{D}\notag\\
													& \hspace{-65pt} =  \sum_{n \geq 0}\sum_{\Delta \geq 0}  \sum_{m \geq 0}\left(\frac{q}{p_1}\right)^{ n+\Delta + n+ m +\floor{-\alpha\Delta} + h}\left(\frac{k}{p_2}\right)^{n}\left(\frac{k}{p_3}\right)^{\Delta + n}\left(\frac{2(d-k)}{p_4}\right)^{m}\notag\\
													&\hspace{-65pt}  = \left(\frac{q}{p_1}\right)^h \sum_{n \geq 0} \left(\frac{q^2k^2}{p_1^2p_2p_3}\right)^n\sum_{\Delta \geq 0}\left(\frac{q}{p_1}\right)^{\Delta + \floor{-\alpha\Delta}}\left(\frac{k}{p_3}\right)^{\Delta}\sum_{m \geq 0}\left(\frac{2(d-k)q}{p_1p_4}\right)^{m}\notag\\
													& \hspace{-65pt} \leq \left(\frac{q}{p_1}\right)^h \sum_{n \geq 0} \left(\frac{q^2k^2}{p_1^2p_2p_3}\right)^n\sum_{\Delta \geq 0}\left(\frac{q}{p_1}\right)^{\Delta(1-\alpha)-1}\left(\frac{k}{p_3}\right)^{\Delta}\sum_{m \geq 0}\left(\frac{2(d-k)q}{p_1p_4}\right)^{m} \notag\\
													& \hspace{-65pt}  = \left(\frac{q}{p_1}\right)^{h-1} \sum_{n \geq 0} \left(\frac{q^2k^2}{p_1^2p_2p_3}\right)^n\sum_{\Delta \geq 0}\left(\frac{q^{1-\alpha}k}{p_1^{1-\alpha}p_3}\right)^{\Delta}\sum_{m \geq 0}\left(\frac{2(d-k)q}{p_1p_4}\right)^{m}. \label{eq:EstReachability}
\end{align}
Now note that if 
\begin{align}\label{eq:boundq}
 q  < \min\left\{\frac{1}{k} p_1 \sqrt{p_2 p_3}\, ,\, p_1 \Big( \frac{p_3}{k}\Big)^{\frac 1{1-\alpha}}\, ,\, \frac{p_1p_4}{2(d-k)}\right\}
\end{align}
then all sums in \eqref{eq:EstReachability} converge and $q/p_1<1$. Thus 
\begin{align*}
 \E_p [\vert \mathcal L_{\eta}^{\alpha, d}(h) \vert ]	& = \sum_{y \in L^{\alpha, d}_{\eta}(h)}	\mathbb P_p(0 \rightarrowtail y) \\
							& \leq 2^d\left(\frac{q}{p_1}\right)^{h-1} \frac{1}{1-\frac{q^2k^2}{p_1^2p_2p_3}} \frac{1}{1-\frac{q^{1-\alpha}k}{p_1^{1-\alpha}p_3}} \frac{1}{1-\frac{2(d-k)}{p_1p_4}}
\end{align*}
and with 
\begin{align*}
 \delta=\delta(q,p_1) & := \frac{q}{p_1} \qquad \text{and} \\
  \qquad C=C(\alpha,d,k,q,p_1,p_2,p_3,p_4) & := 2^d\frac{p_1^2p_2p_3}{p_1^2p_2p_3-q^2k^2} \frac{p_1^{1-\alpha}p_3}{p_1^{1-\alpha}p_3-q^{1-\alpha}k}\frac{p_1p_4}{p_1p_4-2(d-k)q}
\end{align*}
we obtain the claim in \eqref{eq:tailEst}. Lemma \ref{lem:surfaceExistence} then guarantees the existence of an open Lipschitz surface for $q$ as in \eqref{eq:boundq} which completes the proof.
 \end{proof}

 Depending on our choice of the parameters $p_1, p_2, p_3, p_4$ we now obtain different bounds for the critical probability leading to the results of Propositions \ref{prop:allAlphaUBp} and  \ref{prop:qLBkDepOnd}.
 

\begin{proof}[Proof of Proposition \ref{prop:allAlphaUBp}]

In order to obtain Proposition \ref{prop:allAlphaUBp} set
\begin{align*}
 p_1 = \frac{1}{2} \qquad \text{and} \qquad p_2=p_3=\frac{1}{4} -\frac{1}{2}p_4.
\end{align*}
Note that since we consider the case of $k=d,$  the last term on the right-hand side of \eqref{eq:LBqWithp_i} 
is infinite and hence irrelevant. Comparing the first two terms on the right-hand side of \eqref{eq:LBqWithp_i}, one can easily see that the second is the dominating one. 
Thus, taking $p_4 \downarrow 0$,  from \eqref{eq:LBqWithp_i} we  can  deduce the validity of Proposition \ref{prop:allAlphaUBp}.
\end{proof}


\begin{proof}[Proof of Proposition \ref{prop:qLBkDepOnd}]

\textit{(a)} Assume $\varphi(d) \in o(d^{1-\alpha})$ as $ d\rightarrow \infty$. Then for any fixed choice of $p_1, \ldots, p_4$,
as $d \rightarrow \infty$  the last term on the right-hand side of \eqref{eq:LBqWithp_i} is the minimal one
 and thus determines the lower bound for the critical probability given in \eqref{eq:LBqWithp_i}. For every $\varepsilon >0$, choosing $p_1 = 1/2, p_2=p_3=\varepsilon/2$ and $p_4 = 1/2-\varepsilon$, we get
\begin{align*}
 \liminf_{d \rightarrow \infty} q_L(\alpha, d, \varphi(d)) d\geq \frac{1}{2^2}\left(\frac{1}{2}-\varepsilon\right).
\end{align*}
Since this is true for any $\varepsilon >0$, the claim follows.

\textit{(b)} Now consider the case that
for some $c \in [0,1]$ and $\alpha >0$
one has $\varphi(d) \sim c d^{1-\alpha}$ as $d\rightarrow \infty$. Then the second and third term on the right-hand side of \eqref{eq:LBqWithp_i} are asymptotically equivalent and smaller than 
the first term. Hence, they dictate the bound. The claim then holds for any feasible choice of $p_1,\ldots, p_4$ and
\begin{align*}
 C(\alpha,c):= \min\left\{p_1 \left(\frac{p_3}{c}\right)^{\frac{1}{1-\alpha}}, \frac{1}{2}\frac{p_1 p_4}{1-c}\right\}.
\end{align*}
For $\alpha = 0$ we have to take into consideration all three terms of the right-hand side of \eqref{eq:LBqWithp_i}, and thus obtain the claim with 
\begin{align*}
 C(0,c):= \min\left\{\frac{p_1 \sqrt{p_2p_3}}{c} , p_1 \frac{p_3}{c}, \frac{1}{2}\frac{p_1 p_4}{1-c}\right\}.
\end{align*}

\textit{(c)} Now assume that for some $c \in (0,1]$ one has
 $\varphi(d) \sim c d$ as $d\rightarrow \infty$. In this case, the second term on the right-hand side of
\eqref{eq:LBqWithp_i} is  the asymptotically decisive  contribution. Again, for any $\varepsilon >0$, choosing
\begin{align*}
 p_1 = \frac{1-\alpha}{2-\alpha} - 2\varepsilon, \quad p_2 = p_4 = \varepsilon, \quad \text{and} \quad  p_3 = 1-\frac{1-\alpha}{2-\alpha} = \frac{1}{2-\alpha}
\end{align*}
yields
\begin{align*}
 \liminf_{d \rightarrow \infty} q_L(\alpha, d, \varphi(d)) d^{\frac{1}{1-\alpha}} \geq \left(\frac{1-\alpha}{2-\alpha} - 2\varepsilon\right)\left(\frac{1}{2-\alpha}\frac{1}{c}\right)^{\frac{1}{1-\alpha}}.
\end{align*}
Since $\varepsilon$ was arbitrary, 
\begin{align*}
 \liminf_{d \rightarrow \infty} q_L(\alpha, d, \varphi(d)) d^{\frac{1}{1-\alpha}} &  \geq \frac{1-\alpha}{2-\alpha}\left(\frac{1}{2-\alpha}\frac{1}{c}\right)^{\frac{1}{1-\alpha}}\\
										  & = (1-\alpha)\left(1-\frac{1-\alpha}{2-\alpha} \right)^{\frac{2-\alpha}{1-\alpha}}\left(\frac{1}{c}\right)^{\frac{1}{1-\alpha}} \geq (1-\alpha)\frac{1}{4}\left(\frac{1}{c}\right)^{\frac{1}{1-\alpha}}.
\end{align*}
\end{proof}


The next step is to prove Proposition \ref{prop:lbq_alpha}.
\begin{proof}[Proof of Proposition \ref{prop:lbq_alpha}]
 We will again want to apply Lemma \ref{lem:surfaceExistence}. In order to derive an upper bound for the expectation in \eqref{eq:summability}, instead of 
 directly looking at $\lambda$-paths, we will consider a coarse-grained version of them and estimate the probability of these paths reaching a certain height. 
  The reason for coarse-graining is the following: if $q$ is approximately equal to $q_L(\alpha,d,k)$, then an admissible $\lambda$-path starting in 0 (say) will on average pick up at most $1-\alpha$ closed sites per horizontal step and if $q$ is slightly above  $q_L(\alpha,d,k)$, then such a path will certainly exist. When $\alpha$ is very close to one, then  the average number of sites which such a path visits between two successive visits of closed sites
will be of the order $(1-\alpha)^{-1}$  (which is large). If $d \ge 2$, then there will automatically be lots of admissible $\lambda$-paths visiting exactly the same closed sites (in the same order) but taking different routes in between
successive visits to closed sites, the factor increasing to infinity as $\alpha$ approaches 1. This means that estimating the probability that there exists an admissible $\lambda$-path (with a certain property) by the expected number of
such paths (via Markov's inequality) becomes very poor when $\alpha$ is close to 1. Therefore, we will define larger boxes in $\Z^{d+1}$ and define equivalence classes of paths by just observing the sequence of larger boxes they visit. The boxes will then be tuned such that the number of closed sites inside a box is of order one. 

 Recall that w.l.o.g. we assume $\eta_i \in \{0,1\}$, $i = 1, \ldots, d$. To facilitate
 reading, we have structured the proof into three steps.
 
 \noindent \textit{\uline{Step 1:} Coarse-grained $\lambda$-paths.} In order to define the abovementioned paths we partition $\Z^{d+1}$ by dividing $\R^{d+1}$ into boxes as illustrated in Figure \ref{fig:only}:
Define 
\begin{align*}
 B^{\alpha,d,\eta}_0:= \Big\{ r \in \R^{d+1} \, \mid \,& \forall i =1, \ldots, d:\; \big ( 
 \eta_i = 0 \Rightarrow r_i \in [0,1) \big) \land \big(\eta_i \neq 0 \Rightarrow  r_i \in [0, ({1-\alpha})^{-1} ) \big),\\
					      &  \qquad \qquad r_{d+1} \in \Big( \alpha \sum_{i=1}^d \eta_ir_i -1,\alpha \sum_{i=1}^d \eta_ir_i\Big]\Big\}
\end{align*}
and likewise for $a \in \Z^{d+1}$ set $B^{\alpha,d,\eta}_a:= B^{\alpha,d,\eta}_0 + v(a)$, where
\begin{align*}
 v(a):	& = \sum_{i\,:\, \eta_i=0} a_ie_i + \sum_{i\,:\, \eta_i\neq0} a_i\frac{1}{1-\alpha}(e_i + \alpha\eta_ie_{d+1}) + a_{d+1}e_{d+1}\\
	& = \sum_{i\,:\, \eta_i=0} a_ie_i + \sum_{i\,:\, \eta_i\neq0} a_i\frac{1}{1-\alpha}e_i + \left(\sum_{i \, : \, \eta_i\neq0}a_i\frac{\alpha}{1-\alpha}\eta_i + a_{d+1}\right)e_{d+1}.
\end{align*}
\begin{SCfigure} \label{fig:only}
\includegraphics[width=.7\textwidth]{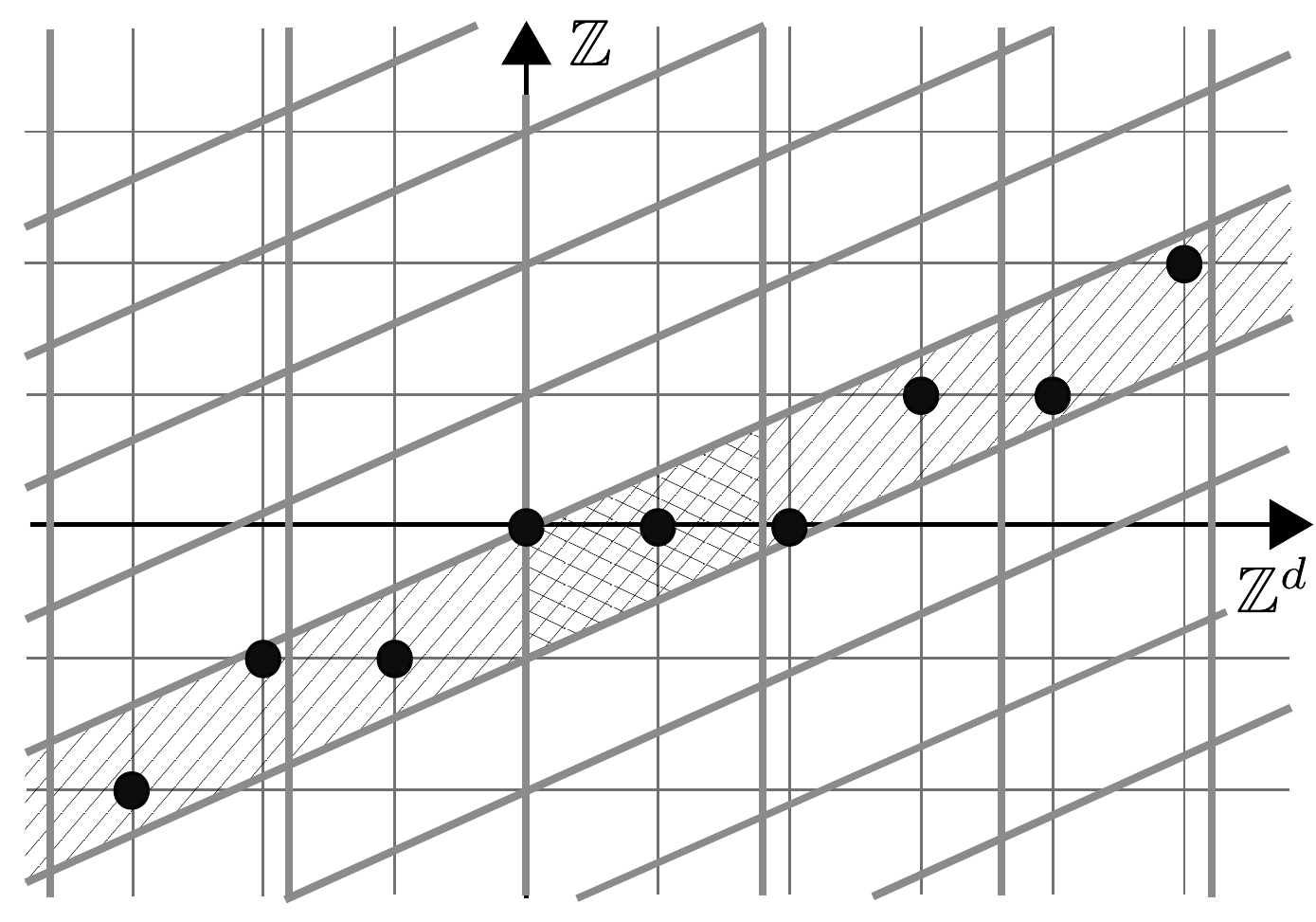}
\caption{$L^{\alpha,d}_{\eta}$ is marked by the black dots and the corresponding coarse-grained boxes are hatched. $B^{\alpha,d,\eta}_0$ is double hatched.}
\end{SCfigure}
Note that these boxes are translations of $B^{\alpha,d,\eta}_0$ shifted either in the direction of $e_{d+1}$ or parallel to the inclination of $L^{\alpha,d}_{\eta}$ and are such that $\Z^{d+1}=\bigcup_{a \in \Z^{d+1}} \big( B^{\alpha,d,\eta}_a \cap \Z^{d+1}\big),$ where the union is over disjoint sets. For any $y\in \Z^d$ the coordinates of the box it is contained in are given by $a(y)\in \Z^{d+1}$ as
\begin{align*}
 a_i(y):=\begin{cases}
          y_i,	& i=1,\ldots,d, \eta_i =0,\\
          \lfloor (1-\alpha)y_i\rfloor,	& i=1,\ldots,d, \eta_i \neq0,\\
          y_{d+1}-\lfloor \alpha \sum_{i=1}^d\eta_iy_i\rfloor,	& i=d+1.
         \end{cases}
\end{align*}
We will refer to these as the \emph{coarse-grained coordinates}. Note that they describe the position of the boxes relative to $L^{\alpha,d}_{\eta}$. Note that for $y \in \Z^d$ the $(d+1)$-st coordinate of its coarse-grained coordinates $a(y)$ gives its height (or distance in the 
$(d+1)$-st coordinate) relative to $L^{\alpha, d}_\eta$.
Since $\alpha,d$ and $\eta$ are fixed for this proof, we will often drop the superscripts for the sake
of better readability. With the above partition
of $\Z^{d+1}$ at hand,
 we can now define coarse-grained $\lambda$-paths.
 A \emph{coarse-grained $\lambda$-path} is any path that takes values in $\bigcup_{a \in \Z^{d+1}} \{B_a\},$
such that it can go from $B_a$ to  $B_{a'}$ in one  time step if and only if 
\begin{align}
 a'-a \in 	&\; \{ e_{d+1}\} \cup \{-\eta_ie_i\mid i=1,\ldots,d, \eta_i\neq0 \} \label{eq:cgsteps}\\
		&  \qquad\cup \{ - e_{d+1}\}\cup \{\pm e_i - e_{d+1}\mid i=1,\ldots,d\} \cup \{\eta_ie_i -2e_{d+1} \mid i=1,\ldots,d, \eta_i\neq0 \}. \notag
\end{align}
In particular, if we sample a standard $\lambda$-path only on the boxes $\{B_a\},$ $a \in \Z^{d+1}$, it visits,
then this supplies us with a coarse-grained $\lambda$-path (however, there might be coarse-grained $\lambda$-paths that
cannot be obtained by this sampling procedure).
We call a box $B_{a}$  \emph{closed} (with respect to $\omega$) if and only if $\omega(x) = 0$ for at least one $x \in B_a.$
Similarly to the case of $\lambda$-paths, we will call a coarse-grained $\lambda$-path \emph{admissible} if for each of its 
\emph{upward steps}, i.e., those steps for which $a'-a = e_{d+1}$, the box $B_{a'}$ 
is closed. Now since the above sampling procedure maps admissible $\lambda$-paths to
admissible coarse-grained $\lambda$-paths, the existence of an admissible $\lambda$-path from some $x \in \Z^{d+1}$ to $y\in \Z^{d+1}$ implies the existence of an admissible coarse-grained $\lambda$-path from $B_{a(x)}$ to $B_{a(y)}$.  We
therefore investigate the behavior of these coarse-grained $\lambda$-paths more closely.

\noindent \textit{ \uline{Step 2:} An estimate for coarse-grained $\lambda$-paths.}
 Recalling \eqref{eq:cgsteps}, note that there is only one kind of step in a coarse-grained $\lambda$-path that will not 
 change its height relative to $L_\eta^{\alpha, d}$, i.e., its coarse-grained coordinate in the $(d+1)$-st dimension, namely those of the form $ -\eta_ie_i$ with $i$ such that $\eta_i\neq0$. Use $\mathsf{CG}(M)$ to denote the set of all coarse-grained $\lambda$-paths starting with $B_0$ of length $M\in \N$ whose endpoint, i.e. its last box, is above or intersects
 $L^{\alpha,d}_{\eta}$. For  $\pi \in \mathsf{CG}(M)$, use $U=U(\pi)$ to denote the number of its \lq up\rq\,-steps, i.e., those steps that increase the $(d+1)$-st coarse-grained coordinate. Similarly, use $D=D(\pi)$ to denote the number of steps that decrease the $(d+1)$-st
  coarse-grained coordinate (possibly by more than 1) and $D^i_0=D^i_0(\pi)$ the number of steps in each dimension $i=1,\ldots,d$, that do \emph{not} alter the 
  $(d+1)$-st coarse-grained coordinate. Due to the natural restrictions on the movements, $D^i_0 = 0$ for any $i$ such that $\eta_i =0$. We can now make the following observation: In order for $\pi$ to end in a box above or intersecting
  $L^{\alpha,d}_{\eta}$, we necessarily have 
$$
U \geq D.
$$
 In addition, observe that due to the length of the boxes in the corresponding directions being $1/(1-\alpha)$, between two steps of type $D^i_0$ (for the same $i$) there needs to be at least one step of type $D$ or $U$ (not $D^j_0, j\neq i$). This implies that 
 $$
 D^i_0 \leq D+U+1.
 $$
  Therefore, for a coarse-grained $\lambda$-path $\pi \in \mathsf{CG}(M)$, recalling that it ends above 
  or intersecting $L^{\alpha,d}_{\eta}$,
\begin{align} \label{eq:estU}
\begin{split}
 M 	& = U + D + \sum_{i=1}^{d} D^i_0  \leq 2U + \Vert \eta\Vert_1 (2 U + 1) = 2U(k+1) + k \\ 
 \Longleftrightarrow \qquad U	& \geq \frac{M-k}{2(k+1)}.
  \end{split}
\end{align}
Thus, we will now  estimate the probability of the event on the right-hand side in the above display. Write $m(\pi)$ for the number of \emph{distinct} boxes visited by a path $\pi \in \mathsf{CG}(M)$. Then the exponential Chebychev inequality yields for any $\beta >0$ and 
 $\gamma \in (0,1)$ that
 
\begin{align} \label{eq:EstClosedSitesInPath}
\P_p(& \text{there exists }  \pi \in \mathsf{CG}(M) 
\text{ whose boxes contain at least } \gamma M \text{ closed sites})\notag\\
	& \leq \sum_{ \pi \in \mathsf{CG}(M)} \P_p(\text{boxes of $\pi$ contain at least } \gamma M \text{ closed sites})\notag\\
	& \leq \sum_{ \pi \in \mathsf{CG}(M)} \frac{1}{\exp(\beta \gamma M)} \E_p[\exp(\beta ( \# \text{ of closed sites in boxes of } \pi))] \notag\\
	& = \sum_{ \pi \in \mathsf{CG}(M)} \frac{1}{\exp(\beta \gamma M)} \E_p[\exp(\beta ( \# \text{ of closed sites in $m(\pi)$ distinct boxes}))] \notag\\ 
	& = \sum_{ \pi \in \mathsf{CG}(M)} \frac{1}{\exp(\beta \gamma M)} (\E_p[\exp(\beta ( \# \text{ of closed sites in } B_0))])^{m(\pi)} \notag\\
	& = \sum_{ \pi \in \mathsf{CG}(M)} \frac{1}{\exp(\beta \gamma M)} (\exp(\beta)q + (1-q))^{\lceil \frac{1}{1-\alpha}\rceil^km(\pi)} \notag\\
	& \leq \sum_{ \pi \in \mathsf{CG}(M)} \frac{1}{\exp(\beta \gamma M)} (\exp(\beta)q + (1-q))^{\lceil \frac{1}{1-\alpha}\rceil^kM} \notag\\
	& \leq (2(2d+1))^M \frac{1}{\exp(\beta \gamma M)} (\underbrace{\exp(\beta)q + (1-q)}_{ \leq \exp(q (\exp(\beta)-1))})^{\lceil \frac{1}{1-\alpha}\rceil^kM} \notag\\
	& \leq \exp \Big (M \Big (\log(4d+2)- \beta\gamma + q(\exp(\beta)-1)\left(\frac{2-\alpha}{1-\alpha}\right)^k \Big ) \Big),
\end{align}
where in the penultimate inequality we estimated the total number of coarse-grained $\lambda$-paths of length $M$ by $(2(2d+1))^M.$
Observe that, choosing $\beta = \frac{1+\epsilon}{\gamma}\log(4d+2)$ for some $\epsilon >0$ the expression inside the exponential is negative if, and only if, 
\begin{align}\label{eq:uglyboundonq}
  -\epsilon\log(4d+2) +& q(\exp(\frac{1+\epsilon}{\gamma}\log(4d+2))-1)  \left(\frac{2-\alpha}{1-\alpha}\right)^k  < 0 \notag\\
    \Leftrightarrow \qquad q & < \frac{\epsilon \log(4d+2)}{\exp((1+\epsilon)\gamma^{-1}\log(4d+2))-1}\left(\frac{1-\alpha}{2-\alpha}\right)^k.
\end{align}

\noindent\textit{\uline{Step 3:} Returning to $\lambda$-paths.}
In order to apply Lemma \ref{lem:surfaceExistence} we need to estimate the probability of reaching a site $y \in L^{\alpha,d}_{\eta}(h)$ with an admissible $\lambda$-path. Recall that coarse-grained  $\lambda$-paths were defined in such a way that
 the existence of an admissible $\lambda$-path from $0 \in \Z^{d+1}$ to $y \in \Z^{d+1}$ implies the existence of an admissible coarse-grained $\lambda$-path from $B_0$ to $B_{a(y)}$. This path then has length $M$ at least $\Vert a(y) \Vert_1$ and thus
\begin{align*}
 M 	& \geq \Vert a(y) \Vert_1 \\
	& \geq \sum_{i=1}^d \vert a_i(y)\vert + h\\
	& = \sum_{i\,:\, \eta_i=0} \vert y_i \vert + \sum_{i\,:\, \eta_i \neq 0} \vert \lfloor(1-\alpha)y_i\rfloor \vert +h\\
	& \geq \sum_{i\,:\, \eta_i=0} \vert y_i \vert + \sum_{i\,:\, \eta_i \neq 0} ((1-\alpha)\vert y_i \vert - 1) +h\\ 
	& \geq (1-\alpha)\Vert \bar y \Vert_1 - k + h.
\end{align*}
Therefore, for any $h \in \N$ and $y \in L^{\alpha,d}_{\eta}(h)$ using \eqref{eq:estU} in the third step,
\begin{align*}
 \P_p(&0 \rightarrowtail y)	 \leq \P_p(\text{there exists an admissible coarse-grained $\lambda$-path from $B_0$ to $B_{a(y)}$})\\
				& \leq \P_p(\text{there exists } \pi \in \mathsf{CG}( (1-\alpha)\Vert \bar y \Vert_1 - k + h ) \text{ admissible})\\
				& \leq \P_p(\text{there exists } \pi \in \mathsf{CG}( (1-\alpha)\Vert \bar y \Vert_1 - k + h ) \\
				& \qquad\quad \text{ whose boxes contain at least } \frac{(1-\alpha)\Vert \bar y \Vert_1 - k + h-k}{2(k+1)} \text{ closed sites})\\
				& \leq  \exp\Big(\big((1-\alpha)\Vert \bar y \Vert_1 - k + h\big)\\
				& \qquad \qquad \quad \times \big(-\epsilon\log(4d+2) + q(\exp((1+\epsilon)4(k+1)\log(4d+2))-1) \left(\frac{2-\alpha}{1-\alpha}\right)^k\big)\Big),
\end{align*}
where
 we choose $h \geq 3k$ and set $\gamma := \frac{1}{4(k+1)}$ to apply \eqref{eq:EstClosedSitesInPath} for the last inequality. Assuming 
\begin{align*}
 q  < \underbrace{\frac{\epsilon \log(4d+2)}{\exp((1+\epsilon)4(k+1)\log(4d+2))-1}\frac{1}{2^k}}_{=:C(k,d,\epsilon)}(1-\alpha)^k
\end{align*}
\eqref{eq:uglyboundonq} holds and combining the observations above we can estimate \eqref{eq:summability} by 
 \begin{align*}
 \sum_{y \in L^{\alpha,d}_{\eta}(h)}  \P_p(0 \rightarrowtail y)	& \leq \sum_{y \in L^{\alpha,d}_{\eta}(h)} \exp
 \Big( ((1-\alpha)\Vert \bar y \Vert_1 - k + h)\\
								& \qquad \qquad \times (\underbrace{-\epsilon\log(4d+2) + q(\exp(\frac{1+\epsilon}{\gamma}\log(4d+2)-1) \lceil\frac{1}{1-\alpha}\rceil^k}_{=:\bar c(k,d,\epsilon,\alpha, q)=\bar c < 0 }) \Big) \\
								& = \exp((-k+h)\bar c)\sum_{y \in L^{\alpha,d}_{\eta}(h)}\exp((1-\alpha)\Vert \bar y \Vert_1\bar c)\\
								& \leq \exp((-k+h)\bar c)\underbrace{\sum_{i=1}^{\infty}\exp((1-\alpha)i\bar c)(2d+1)^i}_{<\infty}.
 \end{align*}
Thus
\begin{align*}
  \lim_{h\rightarrow \infty} \E_p[\vert \mathcal L^{\alpha,d}_{\eta}\vert ] = \lim_{h\rightarrow \infty} \sum_{y \in L^{\alpha,d}_{\eta}(h)}  \P_p(0 \rightarrowtail y) = 0.
\end{align*}
Therefore, the assumptions of Lemma \ref{lem:surfaceExistence} hold which implies the existence of an open Lipschitz surface. Hence,
\begin{align*}
 q_L(\alpha,k,d) \geq C(k,d,\epsilon)(1-\alpha)^k.
\end{align*}
Note that for our result, any $\varepsilon >0$ is sufficient. However, the optimal $\varepsilon$ is given by $\varepsilon = \frac{1 + h}{4(k+1)\log(4d+2)}$, where $h$ is such that $-\exp(-1-4(k+1)\log(4d+2))=h\exp(h)$.
\end{proof}


\begin{proof}[Proof of Proposition \ref{prop:D1}]

In order to prove the lower bound for $q_L(\alpha,1,1)$ we show the existence of an open Lipschitz surface for sufficiently small $q$ by analyzing the existence of an admissible $\lambda$-path starting in $L^{\alpha,1}_{(1)}$ reaching the site $(0,h)$ for large $h \in \N_0$. Writing $x \overset{A}{\rightarrowtail} y$ for the event of existence of an admissible $\lambda$-path from $x \in \Z^2$ to $y \in \Z^2$ that only uses sites in the set $A \subseteq \Z^2$, we observe that
\begin{align}\label{eq:est1}
 \P_p(L^{\alpha,1}_{(1)} \rightarrowtail (0,h)) & = \P_p(L^{\alpha,1}_{(1)} \overset{L^{\alpha,1}_{(1), \geq}}{\rightarrowtail} (0,h)) \leq 2 \P_p\big(\bigcup_{n \in \N_0}\big\{(n, \floor{\alpha n}) \overset{L^{\alpha,1}_{(1), \geq}}{\rightarrowtail} (0,h)\big\}\big) \notag \\
						& \leq 2 \sum_{n =0}^{\infty}\P_p((n, \floor{\alpha n}) \overset{L^{\alpha,1}_{(1), \geq}}{\rightarrowtail} (0,h)),
\end{align}
for any $h \in \N_0$. Therefore we need to find suitable upper bounds for the summands. 

A first helpful bound, albeit without the restriction on the space, can be obtained similarly to \eqref{eq:expNrPathUB}. Observe that any $\lambda$-path from $(n, \floor{\alpha n})$ to $(0,h)$ must have made a total of $4k + \ceil{(2-\alpha)n} + h$ steps for some $k \in \N_0$: $n+k$ to the downward left, $k$ to the downward right and $n-\floor{\alpha n} + h + 2k$ upwards. Then, counting the number of admissible $\lambda$-paths under consideration
\begin{align}\label{eq:genBound}
 \P_p((n, \floor{\alpha n}) \overset{L^{\alpha,1}_{(1), \geq}}{\rightarrowtail} (0,h)) 	& \leq  \P_p((n, \floor{\alpha n}) \rightarrowtail (0,h))\\
											&\leq \sum_{k \in \N_0} \binom{2n + h - \floor{\alpha n} + 4k}{n+k, k, n-\floor{\alpha n} + h + 2k} q^{n-\floor{\alpha n} + h + 2k}.
\end{align} 
This upper bounds the terms for small $n$ in \eqref{eq:est1}, but can also be used to obtain an adequate estimate for large $n$. This is, however, more elaborate: For $n \in \N_0$ define
\begin{align*}
 A_n & := \{-n, -(n-1), \ldots, -1, 0, 1, \ldots \}\times \Z,\\
 Y_n & := \max\{ r \in \Z \mid (0,0) \overset{A_n}{\rightarrowtail }(-n,r) \}.
\end{align*}
$Y_n$ is the height of the highest site above $-n$ reachable by an admissible $\lambda$ path started in $0$ under the restriction of using only the sites in $A_n$. Now note that denoting by $\bar Y_0$ a copy of $Y_0$, independent of $(Y_n)_{n \in \N_0}$,
\begin{align} \label{eq:StDom}
 \bar Y_0 \text{ stochastically dominates } Y_{n+1} - (Y_n -1) \text{ under } \P_p(\cdot\mid Y_i, i \leq n)
\end{align}
since the conditioning can be seen as discarding those paths in the construction using any site below the $Y_i, i \leq n$. Therefore,
 a closer study of the distribution of $\bar Y_0$ seems advisable. Using \eqref{eq:genBound},
\begin{align} \label{eq:YdistrUB}
\begin{split}
 \P_p(\bar Y_0 \geq m) 	& \leq \P_p((0,0) \rightarrowtail (0,m)) \\
			& \leq q^m + \sum_{k \in \N}3^{ m + 4k}q^{m + 2k} \\
			& \leq q^m + (3q)^m \frac{(9q)^2}{1-(9q)^2},
\end{split}
\end{align}
for $q < 1/9$. Hence, we can upper bound the expectation
\begin{align*}
 \E_p[\bar Y_0]	& \leq q + \sum_{m = 2}^{\infty} q^m + \frac{(9q)^2}{1-(9q)^2}\sum_{m =1}^{\infty}(3q)^m \\
		& \leq q + Cq^2
\end{align*}
for a suitable $C>0$ and small $q$. As a consequence, assuming $q$ sufficiently small for 
\begin{align}\label{eq:condq}
q + Cq^2 - 1< -\alpha 
\end{align}
 to hold, \eqref{eq:StDom} and a large deviation principle (the required exponential moments exist due to \eqref{eq:YdistrUB})
 yield the existence of $c_1, c_2>0$ such that
\begin{align*}
 \P_p( Y_n \geq -\alpha n)	& \leq c_1\exp( - n c_2).
\end{align*}
Observe that an admissible $\lambda$ path started in some $(n, \floor{\alpha n})$ and reaching $\{0\}\times\N_0$ going only through $L^{\alpha,1}_{(1), \geq}$ has only used sites to right of $\{0\}\times\Z$ until the first time it hits $\{0\}\times\N_0$. Hence,
\begin{align*}
 \P_p((n, \floor{\alpha n}) \overset{L^{\alpha,1}_{(1), \geq}}{\rightarrowtail} (0,h)) & \leq \P_p((n, \floor{\alpha n}) \overset{L^{\alpha,1}_{(1), \geq}}{\rightarrowtail} \{0\}\times \N_0)\\
		      & \leq \P_p(Y_n \geq -n \alpha) \leq c_1\exp( - n c_2).
\end{align*}
This is the last component needed to estimate \eqref{eq:est1} as it allows us to choose $N \in \N$ such that for any $h \in \N$
\begin{align*}
 \sum_{n = N}^{\infty} \P_p((n, \floor{\alpha n}) \overset{L^{\alpha,1}_{(1),\geq}}{\rightarrowtail} (0,h) )\leq \frac{1}{8}.
 \end{align*}
On the other hand, using \eqref{eq:genBound} again, we may now choose $H$ sufficiently large such that for all $h \geq H,$ 
\begin{align*}
 \sum_{n = 0}^{N-1} \P_p((n, \floor{\alpha n}) \overset{L^{\alpha,1}_{(1),\geq}}{\rightarrowtail} (0,h) )\leq \frac{1}{8}.
\end{align*}
Hence, by \eqref{eq:est1} choosing $q$ as in \refeq{eq:condq} implies
\begin{align*}
 \P_p(L^{\alpha,1}_{(1)} \rightarrowtail (0,h)) \leq \frac{1}{2}
\end{align*}
for all $h \geq H$ and thus $q < q_L(\alpha, 1, 1)$. 

The corresponding upper bound is already given by Proposition \ref{prop:UBqAsympAlpha}.
\end{proof}

\subsection{Upper Bounds for $q_L(\alpha, d, k)$} 

It will be useful in this section to consider what we call reversed $\lambda$-paths. A sequence of sites $x_0, x_1, \ldots, x_n \in \Z^{d+1}$ is called an (admissible) \emph{reversed $\lambda$-path}, if $x_n, x_{n-1},x_{n-2}, \ldots, x_0$ is an (admissible) $\lambda$-path in the sense of Definition \ref{def:adLamPa}. 

Furthermore, the proof of Proposition \ref{prop:UBqAsympD} will take advantage of a comparison to so-called $\rho$-percolation, see e.g. \cite{MeZu-93} and \cite{KeSu-00}. Here the setting is that of oriented site-percolation in $\Z^d$, i.e., where in addition to our standard setting
of Bernoulli site percolation we assume the nearest neighbor edges of $\Z^d$ to be oriented in the direction of the positive coordinate vectors
(which is the sense of orientation for the rest of this section). We say that \emph{$\rho$-percolation occurs} for $\omega \in \{0,1\}^{\Z^d}$ if there exists an oriented nearest neighbor path $0=\bar x_0, \bar x_1, \ldots$ in $\Z^d$ starting in the origin, such that 
\begin{align*}
 \liminf_{n \rightarrow \infty} \frac{1}{n} \sum_{i=1}^n (1-\omega(\bar x_i)) \geq \rho.
\end{align*}
Any such path is called a \emph{$\rho$-path}.
The probability of the existence of such a path exhibits a phase transition in the parameter $q$ and the corresponding critical probability is denoted by $q_c(\rho, d)$. Theorem 2 in \cite{KeSu-00} states that for every $\rho \in (0, 1]$,
\begin{align} \label{eq:RhoConst}
 \lim_{d \rightarrow \infty} d^{\frac{1}{\rho}}q_c(\rho, d) = \frac{\theta^\frac{1}{\rho}}{e^{\theta} -1} =:R(\rho),
\end{align}
where $\theta$ is the unique solution to $\theta e^{\theta}/(e^{\theta} -1) = 1/\rho$, and $R(1)=1$. Note that we have interchanged the role of `open' and `closed' (and thus $p$ and $q$) with respect to \cite{KeSu-00} in order to adapt the result to its application in our proof. 

Before turning to the proof of Proposition \ref{prop:UBqAsympD}, we observe a useful property of the critical probability of $\rho$-percolation.


\begin{lemma}[Continuity of $q_c$]\label{lem:rho}
 The critical probability of $\rho$-percolation is continuous in $\rho$, i.e. for any $d \in \N$ the map 
 \begin{align} \label{eq:qcCont}
[0,1) \ni  \rho \mapsto q_c(\rho, d)
 \end{align}
is continuous.
\end{lemma}


\begin{proof}[Proof of Lemma \ref{lem:rho}]
 Since $d$ is fixed and we only consider $\Z^d$ in this proof, the index is dropped for better readability. It is easy to see that the event of $\rho$-percolation also undergoes a phase-transition in $\rho$ (for fixed $q$) and thus we define
\begin{align*} 
 \rho_c(q):= \sup\{\rho \mid \P_{1-q}(\rho\text{-percolation occurs}) = 1\}.
\end{align*}
Note that strict monotonicity of $\rho_c(q)$ for $q \in [0,\bar q]$, where $\bar q := \sup\{ q\mid \rho_c(q) <1\}$, 
would imply the desired continuity of $q_c(\rho)$ on $[0,1)$. In order to prove this strict monotonicity, we will, however, first consider a different quantity: Still in the setting of oriented percolation in $\Z^d$, for any $\omega \in \{0,1\}^{\Z^d}$ 
let
\begin{align*}
 Y_{0,n}(\omega) := \max \Big\{\, r \in \N_0 \mid \exists  \text{ directed nearest neighbor path }& 0 = x_0, x_1, \ldots, x_{n}:\\
						      & \qquad \qquad \, \sum_{i=1}^{n} (1- \omega(x_i)) = r
						      \Big\},
\end{align*}
and denote by $\hat X_n$ the site with the lowest lexicographical order that is the endpoint of such a directed nearest neighbor path on which the value of $Y_{0,n}$ is attained. Then, for $m\geq n$  define
\begin{align*}
 Y_{n,m}(\omega) := \max \Big\{\, r \in \N_0 \mid \exists  \text{ directed nearest neighbor path }& \hat X_n = x_0, x_1, \ldots, x_{m-n}:\\
						      & \qquad \qquad \, \sum_{i=1}^{m-n} (1- \omega(x_i)) = r
						      \Big\}.
\end{align*}

By the Subadditive Ergodic Theorem (see e.g. \cite{Durrett}, Theorem 6.6.1)
 the sequence $(Y_{0,n}/n)_{n \in \N}$ converges $\P_{1-q}$-a.s. and in $L^1(\P_{1-q})$ to a (deterministic) limit that we denote by $\gamma(q)$. In fact, 
 \begin{equation} \label{eq:limitEquality}
\gamma (q) = \rho_c(q).
 \end{equation}
 To see this, fix $q \in (0,1)$ and choose $\rho < \rho_c(q)$. 
Then for $\P_{1-q}$-almost any $\omega \in \{0,1\}^{\Z^d}$ there exists an oriented nearest neighbor path $X_1(\omega), X_2(\omega), \ldots$ such that 
\begin{align*}
 \rho \leq \liminf_{n \rightarrow \infty} \frac{1}{n} \sum_{i=1}^n (1-\omega(X_i(\omega))).
\end{align*}
Since by definition $ \sum_{i=1}^n (1-\omega(X_i(\omega))) \leq Y_{0,n}(\omega)$
for $\P_{1-q}$-almost all $\omega\in \{0,1\}^{\Z^d}$ and $n \in \N$, taking the limes inferior on both sites gives $\rho \leq \gamma(q)$, which implies $\rho_c(q) \leq \gamma(q)$. To prove the converse inequality, choose, for any $\varepsilon >0$ an $N \in \N$ such that $\frac{1}{N}\E_{1-q}[Y_{0,N}] \geq \gamma(q)-\varepsilon$. For any $\omega \in \{0,1\}^{\Z^d}$ let $X_1(\omega),X_2(\omega),\ldots, X_N(\omega)$ be an (oriented nearest neighbor) path, such that $Y_{0,N} = \sum_{i=1}^{N} (1- \omega(X_i(\omega)))$. Using i.i.d. copies of $(X_1, \ldots, X_N)$, one can construct an infinite oriented nearest neighbor path $(\widetilde X_i)_{i \in \N_0}$ with the property that by the law of large numbers
\begin{align*}
 \lim_{n \rightarrow \infty} \frac{1}{n} \sum_{i=1}^n \widetilde X_i = \frac{1}{N}\E_{1-q}[Y_{0,N}] \geq \gamma(q)-\varepsilon \qquad \P_{1-q}\text{-a.s.}.
\end{align*}
Thus $\gamma(q) - \varepsilon \leq \rho_c(q)$ and since $\varepsilon$ was arbitrary, $\gamma(q) \leq \rho_c(q)$, which in combination with the above establishes \eqref{eq:limitEquality}.

The strict monotonicity of $\gamma(\cdot)$ (and thus $q_c(\cdot)$) can now be proven through a suitable coupling argument. Denote by $\mathcal U_{[0,1]}$ the uniform measure on the interval $[0,1]$ and define $\mu:= \mathcal U_{[0,1]}^{\otimes \Z^{d+1}}$ as the product measure on the space $\mathcal W := [0,1]^{\Z_d}$. For any $w \in \mathcal W$, $q \in (0,1)$ and $n \in \N_0$ define
\begin{align*}
 Y^q_n(w): = \max \Big \{r \in \N_0 \mid \exists  \text{ directed nearest neighbor path }& 0 = x_0, x_1, \ldots, x_{n}:\\
						      & \qquad \qquad \, \sum_{i=}^{n} \ind{1}_{[0,q]}(w(x_i)) = r
						      \Big \}.
\end{align*}
Observe that $\mathcal L_{\mu} ((Y^q_n)_{n \in \N_0}) = \mathcal L_{\P_{1-q}}((Y_{0,n})_{n \in \N_0})$, where $\mathcal L_{\nu}$ denotes the law with respect to the measure $\nu$. Therefore 
\begin{align*}
 \lim_{n \rightarrow \infty} \frac{1}{n} Y^q_n = \gamma(q) \qquad \mu\text{-a.s. and in } L^1(\mu).
\end{align*}
As before, for any $q \in (0,1)$, $w \in \mathcal W$ and $n \in \N_0$, let $X^{q,n}_1(w), \ldots, X^{q,n}_n(w)$ be an oriented nearest neighbor path such that $Y^q_n = \sum_{i=1}^n \ind{1}_{[0,q]}(w(X^{q,n}_i(w)))$. Choose $0\leq q < q' \leq \bar q$, then
\begin{align} \label{eq:DiffInY}
 Y^{q'}_n = \sum_{i=1}^n \ind{1}_{[0,q']}(w(X^{q',n}_i(w))) & \geq \sum_{i=1}^n \ind{1}_{[0,q']}(w(X^{q,n}_i(w))) \notag\\
							    & = Y^q_n + \sum_{i=1}^n \ind{1}_{[q,q']}(w(X^{q,n}_i(w))).
\end{align}
Set $\mathcal F_q := \sigma(w \mapsto \ind{1}_{[0,q]} (w(x)) \mid x \in \Z^d)$. Then, obviously, the $Y^q_n$ are $\mathcal F_q$-measurable and the $\ind{1}_{[0,q]}(w(X^{q,n}_i(w)))$, $1 \le i \le n,$ are independent given $\mathcal F_q$. In addition,
\begin{align*}
 \mu(\ind{1}_{[q,q']}(w(X^{q,n}_i(w))) = 1 \mid \mathcal F_q) = \frac{q'-q}{1-q} \ind{1}_{\{w(X^{q,n}_i(w))>q\}}.
\end{align*}
Thus using \eqref{eq:DiffInY} we obtain
\begin{align*}
 \E_{\mu}[Y^{q'}_n-Y^q_n \mid \mathcal F_q] \geq \E_{\mu}
 \Big [ \sum_{i=1}^n \ind{1}_{[q,q']}(w(X^{q,n}_i(w)))\mid \mathcal F_q \Big] = \left( n - Y^q_n\right)\frac{q'-q}{1-q},
\end{align*}
where $\E_{\mu}$ denotes the expectation with respect to $\mu$. Using the $L_1(\mu)$ convergence
\begin{align*}
 \gamma(q') - \gamma(q) & = \lim_{n \rightarrow \infty} \E_{\mu}
 \Big[ \E_{\mu} \Big[\frac{1}{n}\left(Y^{q'}_n - Y^q_n\right)\mid \mathcal F_q \Big ] \Big] \\
			& \geq \lim_{n \rightarrow \infty} \E_{\mu} \Big[ \frac{1}{n} \left(n-Y^q_n\right) \frac{q'-q}{1-q} \Big]\\
			& = (1-\gamma(q))\frac{q'-q}{1-q}
\end{align*}
and the right-hand side is positive, since $\gamma(q)=\rho_c(q) < 1$ for $q < \bar q$. This shows the strict monotonicity
of the function $\rho_c$  on  $[0,\overline q]$ and hence implies \eqref{eq:qcCont}.
\end{proof}


\begin{proof}[Proof of Proposition \ref{prop:UBqAsympD}]

We will compare $\rho$-paths in $\Z^d$ with reversed admissible $\lambda$-paths in $\Z^{d+1}$.
 To this end define for any $\omega \in \{0,1\}^{\Z^{d+1}}$ and $\bar x \in \Z^d$ the quantity
\begin{align*}
 H_{\omega}(\bar x): = \min\Big\{h \in \N_0 \mid \exists & \text{ an oriented 
 nearest neighbor path } 0=\bar x_0, \ldots, \bar x_m=\bar x \in \Z^d, \\
						     & \text{ and a sequence } 0=h_0, \ldots, h_m = h \in \N_0 \text{ s.t. } \\
						     & \qquad \qquad h_{i+1} = \begin{cases}
									    h_i, & \text{ if }\omega(\bar x_i, h_i) = 0,\\
									    h_i+1, & \text{ otherwise}.
									\end{cases}  \quad \Big\}.
\end{align*}
A second's thought reveals that this map is defined in such
a way that there is an admissible $\lambda$-path from $(\bar x, H_{\omega}(\bar x))$ to the origin,
 which takes advantage of many closed sites in the
 configuration $\omega$. (It is, however, not optimal, as it does not make use of consecutive \lq piled up\rq\ closed sites in one step.) With this we can then define a map $T : \{0,1\}^{\Z^{d+1}} \rightarrow \{0,1\}^{\Z^d}$ as
\begin{align*}
 (T(\omega)) (\bar x) := \begin{cases}
                     \omega(\bar x, H_{\omega}(\bar x)), & \text{ if } \bar x \in \N_0^d,\\
                     \omega(\bar x, 0), & \text{ otherwise.}
                    \end{cases}
\end{align*}
The purpose of $T$ is to map a configuration $\omega \in \{0,1\}^{\Z^{d+1}}$ to a configuration $\bar \omega \in \{0,1\}^{\Z^d}$, 
for which there exists an oriented path picking up almost as many closed sites as the oriented reversed admissible $\lambda$-path in $\omega$ with lowest $(d+1)$-st coordinate. In order to be more precise, we add an index to the probability measure used to indicate the space it is defined on. 
I.e., $\P_{p,d}$ will denote the Bernoulli product-measure on $\Z^d$ with parameter $p$. Since the value of $H(\bar x)$ only depends on the state of the sites $\bar y \in \N_0^d$ with $\Vert \bar y \Vert < \Vert \bar x \Vert$, $\P_{p,d+1} \circ T^{-1} = \P_{p,d}$. Thus, if $q > q_c(\rho,d)$, we have that
\begin{align}\label{eq:rhopercas}
 1 	& = \P_{p,d}(\text{$\rho$-percolation occurs}) \notag\\
	& = \P_{p,d+1} \Big(\text{there exists an admissible reversed $\lambda$-path } 0 = (\bar x_0, h_0), (\bar x_1, h_1), \dots \\
	& \hspace{300pt} \text{ s.t. } \limsup_{n \rightarrow \infty} \frac{1}{n} h_n \leq 1-\rho \Big). \notag
\end{align}
Now choose $\rho > 1-\alpha$ and set $\delta := 1- \rho + (\alpha - (1-\rho))/2 \in ( 1-\rho, \alpha)$. Then \eqref{eq:rhopercas} implies the existence of a (deterministic) $N \in \N$ such that for all $n \geq N,$
\begin{align*}
\P_{p,d+1}\Big(& \text{there exists an admissible reversed $\lambda$-path } \\ 
	    & \qquad 0 = (\bar x_0, h_0), (\bar x_1, h_1), \dots, (\bar x_n, h_n)  \text{ s.t. }  h_n \leq \delta n\Big) \geq \frac{1}{2}.
\end{align*}
Note that if there exists an admissible reversed $\lambda$-path from the origin to some $(\bar x_n, h_n)$ with $h_n \leq \delta n$, then there actually exists an admissible $\lambda$-path from $L^{\alpha,d}_{\eta} - \floor{(\alpha - \delta)n}e_{d+1} $ to the origin. Thus, by translation invariance of $\P_{p,d+1}$, we obtain that
\begin{align*}
 \forall n \geq N:\; \P_{p,d+1} \big( L^{\alpha,d}_{\eta} \rightarrowtail (0, \floor{(\alpha - \delta)n}) \big) \geq \frac{1}{2}
\end{align*}
which, since $\alpha - \delta >0$, implies
\begin{align*}
 \P_{p,d+1}\left( (\mathsf{LIP}^{\alpha,d}_{\eta})^c\right) = \lim_{n \rightarrow \infty} \P_{p,d+1} \big( L^{\alpha,d}_{\eta} \rightarrowtail (0, (\alpha - \delta)n ) \big) \geq \frac{1}{2}.
\end{align*}
By Proposition \ref{prop:equiv} we deduce that $\P_p\left(\mathsf{LIP}^{\alpha,d}_{\eta}\right) = 0$ and
hence $q \geq q_L(\alpha, d,d)$. We have thus shown that for any $\rho > 1-\alpha$ one has $ \;q_c(\rho,d) \geq q_L(\alpha,d,d)$. Since $q_c(\rho,d)$ is continuous in $\rho$ by Lemma \ref{lem:rho},
then the claim follows from \eqref{eq:RhoConst}. 
\end{proof}


\begin{lemma}[Criterion for non-existence of an open Lipschitz surface] \label{lem:CritLSNonExist}
 For any $\alpha >0$, and $d \in \N$ define 
 \begin{align*} 
  T := \inf \{ m \in \N_0 \mid \exists \bar x \in \N_0^d:\; \Vert \bar x \Vert_1 = m \text{ and } (\bar x, \Vert \bar x \Vert_1) \text{ is closed}\}.
 \end{align*}
If for $p \in (0,1)$ one has 
\begin{align}\label{eq:CritLSNonExist}
 \E_p[T] < \frac{1}{1-\alpha},
\end{align}
then $\P$-a.s. there exists no open Lipschitz surface and $q= 1-p \geq q_L(\alpha,d,d)$. 
\end{lemma}
Condition \eqref{eq:CritLSNonExist} has an intuitive interpretation: $1/(1-\alpha)$ is the number of \lq downward-diagonal\rq\ steps a $\lambda$-path can take before decreasing its distance to the plane with inclination $\alpha$ by one. $\E_p[T]$ on the other hand is the expected number of such steps an admissible $\lambda$ path must take before encountering a closed site and thus being able to take an upwards step. \eqref{eq:CritLSNonExist} therefore means that this path will -- on average -- encounter a closed site strictly before decreasing its distance to the plane by one, thus increasing the distance in the long run and preventing the existence of an open Lipschitz surface above it.


\begin{proof}
As in the proof of Proposition \ref{prop:UBqAsympD}, the idea is to construct admissible reversed $\lambda$-paths starting in $0$ such that their endpoints (i.e. the starting points of the respective $\lambda$-paths) are arbitrarily far below $L^{\alpha,d}_{\eta}$. With a simple shifting argument we can then see that the Lipschitz surface would, with probability bounded away from $0$, have to have arbitrarily large height in $0$ and can therefore almost surely not exist.
 
 We begin with the construction of the \emph{reversed} $\lambda$-paths. To this end, set $ X_0:=Y_0:=0$. Let $(\bar z_i)_{i \in \N_0}$ be an ordering of $\N_0^d$ compatible with $\Vert \cdot \Vert_1$ in the sense that $\Vert z_{i+1} \Vert_1 \geq \Vert z_i \Vert_1$, for all $i \in \N_0$. Then define for any $n \in \N_0$,
 \begin{align*}
  \iota_{n+1} 	& := \inf\{i \in \N_0 \mid (\bar z_i, \Vert\bar z_i \Vert_1) + Y_n \text{ is closed}\},\\
  X_{n+1} 	& := (\bar z_{\iota_n}, \Vert \bar z_{\iota_n}\Vert_1), \\
  Y_{n+1}	& := Y_n + X_{n+1} - e_{d+1}.
 \end{align*}
By construction, there always exists an admissible $\lambda$-path from any $Y_n$ to 0. Note also that $(\iota_n)_{n \in \N}$ and $(X_n)_{n \in \N}$ are i.i.d. sequences where $\iota_1$ is geometric on $\N_0$ with parameter $q$ and $\Vert \bar X_1 \Vert_1 = X_1 \cdot e_{d+1}$
is distributed as $T$. 

We are now interested in the height of the starting points of these $\lambda$-paths relative to $L^{\alpha,d}_{\eta}$. This is given by 
\begin{align*}
 H(n)	& := \floor{\alpha \Vert \bar Y_n \Vert_1} - Y_n\cdot e_{d+1}\\
	& = \Big \lfloor \alpha \sum_{j = 1}^n \Vert \bar X_j \Vert_1 \Big \rfloor - \sum_{j=1}^n ( X_n - e_{d+1})\cdot e_{d+1}\\
	& =  \Big \lfloor \alpha \sum_{j = 1}^n \Vert \bar X_j \Vert_1 \Big \rfloor - \sum_{j = 1}^n \Vert \bar X_j \Vert_1 + n.
\end{align*}
 The law of large numbers then yields
 \begin{align*}
  \lim_{n \rightarrow \infty} \frac{1}{n}H(n) = (\alpha - 1)\E_p[T] +1 \qquad \P_p\text{-a.s.}
 \end{align*}
and the right-hand side is strictly negative by assumption. Thus with $\Delta := -\left( (\alpha - 1)\E_p[T] +1 \right)/2 >0$ we have in particular the existence of a deterministic $N \in \N$ such that
\begin{align*}
 \forall n \geq N: \; \P_p( H(n) \leq - \Delta n) \geq \frac{1}{2}.
\end{align*}
Now note that on the event $\{H(n) \leq - \Delta n\}$ there exists an admissible $\lambda$-path starting in $L^{\alpha,d}_{\eta} - \Delta n e_{d+1}$ and reaching 0, since $Y_n$ is below the plane $L^{\alpha,d}_{\eta} - \Delta n e_{d+1}$. Hence, by translation invariance of $\P_p$ we have that
\begin{align*}
 \forall n \geq N: \; \P_p(L^{\alpha,d}_{\eta} \rightarrowtail (0, \Delta n) \geq \frac{1}{2}
\end{align*}
which implies
\begin{align*}
 \P_p\left( (\mathsf{LIP}^{\alpha,d}_{\eta})^c\right) = \lim_{n \rightarrow \infty} \P_p(L^{\alpha,d}_{\eta} \rightarrowtail (0, \Delta n) \geq \frac{1}{2}.
\end{align*}
By Proposition \ref{prop:equiv}, $\P_p\left(\mathsf{LIP}^{\alpha,d}_{\eta}\right) = 0$ and $p \leq p_L(\alpha,d,d)$, i.e., $q \geq q_L(\alpha, d,d)$. 
\end{proof}


\begin{proof}[Proof of Proposition \ref{prop:UBqAsympAlpha}]
 Recall the ordering $(\bar z_i)_{i \in \N_0}$ of $\N_0^d$ compatible with $\Vert \cdot \Vert_1$ from the proof of Lemma \ref{lem:CritLSNonExist} and define the random variable
 \begin{align*}
 \iota_1 :=  \inf\{i \in \N_0 \mid (\bar z_i, \Vert\bar z_i \Vert_1) \text{ is closed}\}, 
 \end{align*}
 which has a geometric distribution on $\N_0$ with parameter $q$. With $B(j):= \{ \bar x \in \N_0^d \mid \Vert \bar x \Vert_1 \leq j\}$ denoting the ball with radius $j \in \N_0$, define the function 
 \begin{align*}
  r(i) := \inf\{j \in \N_0 \mid \vert B(j) \vert - 1 \geq i\}
 \end{align*}
that gives the radius of the smallest ball such that its cardinality (without the origin) is larger than
or equal to a given $i \in \N_0$. Note that $r(\iota_1)$ is distributed as $T$,
 for $T$ defined in Lemma \ref{lem:CritLSNonExist}.
Using
\begin{align*}
 \vert B(j) \vert = \binom{j+d}{d} \geq  \frac{(j+1)^d}{d!}
\end{align*}
we obtain
\begin{align*}
 i \geq \vert B(r(i)-1) \vert \geq  \frac{r(i)^d}{d!}
\end{align*}
 and can thus upper bound the expectation
 \begin{align*}
  \E_p[T] = \E_p[r(\iota_1)] 	& \leq \left( d!\E_p[\iota_1] \right)^{\frac{1}{d}}  \leq \left( d!\left(\frac{1}{q}-1\right)\right)^{\frac{1}{d}},
 \end{align*}
 where we used Jensen's inequality in the first inequality.
The right-hand side is strictly smaller than $1/(1-\alpha)$ if and only if 
\begin{align*}
 q > \frac{d!(1-\alpha)^d}{1+d!(1-\alpha)^d}.
\end{align*}
Thus Lemma \ref{lem:CritLSNonExist} then implies that for such values of $q$ no open Lipschitz surface can exist, i.e. $q \geq q_L(\alpha,d,d)$, and the claim follows.
\end{proof}

\subsection*{Acknowledgement} 
We thank Patrick W. Dondl for helpful suggestions and valuable discussions.

\bibliographystyle{alpha}

\bibliography{Bib}

\end{document}